\documentclass[11pt]{article}
\usepackage[utf8]{inputenc}
\usepackage{amscd,amsmath,amsthm,amssymb,amsfonts}
\usepackage{pstcol,pst-plot,pst-3d}
\usepackage{txfonts}
\usepackage{eucal}
\usepackage{bbm}
\usepackage{appendix}
\usepackage{authblk}
\usepackage{graphicx} 
\usepackage{hyperref}
\hypersetup{
    colorlinks=true,
    linkcolor=blue,
    filecolor=blue,      
    urlcolor=blue,
    citecolor=blue,
    }
\usepackage{algorithm}
\usepackage{mathrsfs}
\usepackage{derivative}
\usepackage[noend]{algorithmic}
\usepackage[a4paper]{geometry}
\usepackage{tikz-cd}
\usepackage{mathtools}
 \usepackage[nottoc]{tocbibind}    
\usepackage{empheq}
\usepackage{graphics}
\usepackage{indentfirst}
\usepackage[
backend=bibtex,
style=numeric-comp,
maxnames=4
]{biblatex}
\usepackage{color}
\usepackage{tabularx,colortbl}
\numberwithin{equation}{section}

\usepackage{fancyhdr}

\usepackage{subcaption}
\DeclareCaptionLabelFormat{custom}
{%
      \textbf{Fig. #2}
}
\DeclareCaptionLabelSeparator{custom}{ }
\DeclareCaptionFormat{custom}
{%
    #1 #2 #3
}
\usepackage{etoolbox}

\makeindex

\def\s{$\mbox{\c{s}}$}

\def\opn#1#2{\def#1{\operatorname{#2}}} 

\opn\Ker{Ker} \opn\Coker{Coker}  \opn\Hom{Hom} \opn\Im{Im}
\opn\End{End} \opn\Aut{Aut} \opn\defect{def} \opn\ord{ord}
\opn\id{id} \opn\dim{dim} \opn\det{det} \opn\tr{tr} \opn\grad{grad} \opn\lcm{lcm}
\opn\min{min} \opn\max{max} 
\opn\Span{Span}   \opn\rang{rang}  \opn\id{id} \opn\Ass{Ass} \opn\Min{Min}
\opn\GL{GL} \opn\SL{SL} \opn\mod{mod} \opn\diag{diag}
\opn\min{min} \opn\sgn{sgn} \opn\ini{in_<}  \opn\Mon{Mon} \opn\LC{LC_<} \opn\Hom{Hom} \opn\Ext{Ext} \opn\gini{gin_{<_{rev}}} \opn\gin{gin_{<}}
\opn\LT{LT_<}
\opn\s{supp} \opn\Tor{Tor} \opn\link{link} \opn\depth{depth} \opn\pd{pd} \opn\reg{reg} 
\newcommand{\compactfundif}{\textup{C}^{1}_{\textup{c}}(\mathbb{R}^{2}_{>0})}

\newcommand{\compactfun}{\textup{C}_{\textup{c}}(\mathbb{R}^{2}_{>0})}
\newcommand{\vanishfun}{\textup{C}_{\textup{0}}(\mathbb{R}^{2}_{>0})}

\newcommand{\supp}{\textup{supp}}
\newcommand{\der}{\textup{d}}

\date{}
\title{Fast fusion in a two-dimensional coagulation model}
\author[a]{Iulia Cristian}
\author[b]{Juan J. L. Vel\'{a}zquez}
\affil[a]{Institute for Applied Mathematics, University of Bonn, Endenicher Allee 60, 53115 Bonn, Germany
\href{mailto:cristian@iam.uni-bonn.de}{cristian@iam.uni-bonn.de}}
\affil[b]{Institute for Applied Mathematics, University of Bonn, Endenicher Allee 60, 53115 Bonn, Germany
\href{mailto:velazquez@iam.uni-bonn.de}{velazquez@iam.uni-bonn.de}}

\begin{document}
\maketitle
\newtheorem{teo}{Theorem}[section]
\newtheorem{ex}[teo]{Example}
\newtheorem{prop}[teo]{Proposition}
\newtheorem{obss}[teo]{Observations}
\newtheorem{cor}[teo]{Corollary}
\newtheorem{lem}[teo]{Lemma}
\newtheorem{prob}[teo]{Problem}
\newtheorem{conj}[teo]{Conjecture}
\newtheorem{exs}[teo]{Examples}
\newtheorem{alg}[teo]{\bf Algorithm}

\theoremstyle{definition}
\newtheorem{defi}[teo]{Definition}

\theoremstyle{remark}
\newtheorem{rmk}[teo]{Remark}
\newtheorem{ass}[teo]{Assumption}




\pagenumbering{arabic}

\begin{abstract}
 In this work, we study a particular system of coagulation equations characterized  by two values, namely volume $v$ and surface area $a$. Compared to the standard one-dimensional models, this model incorporates additional information about the geometry of the particles. We describe the coagulation process as a combination between collision and fusion of particles. We prove that we are able to recover the standard one-dimensional coagulation model when fusion happens quickly and that we are able to recover an equation in which particles interact and form a ramified-like system in time when fusion happens slowly.
\end{abstract}

\vspace{0.5cm}


\tableofcontents

\section{Introduction}
Most of the works on coagulation equations assume that the particles are characterized by a single variable, usually the particle volume (or equivalent quantities like polymer length), see for instance \cite{Niethammer_2012,norris1999,1916ZPhy...17..557S,STEWART}. Nevertheless, other parameters that might provide insight about the geometry or other features of the particles are usually omitted. In a previous work (see \cite{cristian2022coagulation}), we study the mathematical properties of a class of coagulation equations in which the aggregating particles are characterized by two degrees of freedom, namely the volume $v$ and the surface area $a$. This type of models was introduced in \cite{book1, KOCH1990419}. More precisely, the model considered is the following:
\begin{align}
\partial_{t}f(a,v,t)+\partial_{a}[r(a,v)(c_{0}v^{\frac{2}{3}}-a)f(a,v,t)]=\mathbb{K}[f](a,v,t), \textup{  }  c_{0}:=(36\pi)^{\frac{1}{3}},
\label{strongfusioneq}
\end{align}
where 
\begin{align*}
    \mathbb{K}[f](a,v,t):=&\frac{1}{2}\int_{(0,a)\times(0,v)}K(a-a',v-v',a',v')f(a',v',t)f(a-a',v-v',t)\der v'\der a' \\
    &-\int_{(0,\infty)^{2}}K(a,v,a',v')f(a,v,t)f(a',v',t)\der v' \der a'.
\end{align*}

In this model, $f$ is the density of particles in the space of area and volume for any given time $t\geq 0$. The coagulation operator $\mathbb{K}[f]$ is the classical coagulation operator that was introduced by Smoluchowski (see \cite{1916ZPhy...17..557S}) and gives the coagulation rate of particles which evolve according to the following mechanism:
\begin{align*}
(a_{1},v_{1})+(a_{2},v_{2})\longrightarrow (a_{1}+a_{2},v_{1}+v_{2}).
\end{align*}
It is assumed that the particles attach to each other at their contact point and therefore in this way both the total area and volume of the particles involved in the process are preserved.

On the other hand, the fusion term $\partial_{a}[r(a,v)(c_{0}v^{\frac{2}{3}}-a)f(a,v,t)]$ describes an evolution of the particles towards a spherical shape. The dynamics generated by this term preserves the total number and volume of the particles. The term $c_{0}v^{\frac{2}{3}}-a$ indicates that the area of the particles tends to be reduced as long as it is larger than that of a sphere $c_{0}v^{\frac{2}{3}}$ (see Figure \ref{fig1} for a description of the complete coagulation process assumed in (\ref{strongfusioneq})).

\begin{figure}[H]%
\centering
\captionsetup{justification=centering}
\includegraphics[width=0.9\textwidth, height=3.1cm]{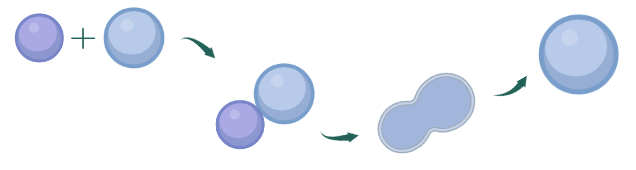}
\caption{Coagulation process: collision of particles followed by fusion}
\label{fig1}
\end{figure}

Additionally, $r(a,v)$ will indicate the fusion rate and describes how quickly the particles evolve towards the spherical shape and thus has units of the inverse of the fusion time. If the fusion kernel $r$ is very large compared with the coagulation rate, we expect that the particles become spherical in very short times. Therefore, it should be possible to approximate the solutions of (\ref{strongfusioneq}) by means of solutions of a coagulation model depending on only the variable $v$, i.e. an one-dimensional coagulation equation. On the contrary, in the particular case when $r\equiv 0$, fusion does not occur and particles attach at contact points forming a ramified-like system in time. Thus, when $r$ is very small compared with the coagulation rate, we can approximate the solutions of (\ref{strongfusioneq}) by means of solutions of a two-dimensional coagulation model without fusion depending on two variables, $a$ and $v$.

More precisely, we analyse the following model as $\Lambda\rightarrow 0$ and as $\Lambda\rightarrow\infty$:
\begin{align}
\partial_{t}f_{\Lambda}(a,v,t)+\frac{1}{\Lambda}\partial_{a}[r(a,v)(c_{0}v^{\frac{2}{3}}-a)f_{\Lambda}(a,v,t)]=\mathbb{K}[f_{\Lambda}](a,v,t).
\label{1 over epsilon strongfusioneq}
\end{align}

 We remark that the particles must satisfy the isoperimetric inequality, therefore the density $f_{\Lambda}$ should be supported in the region where $\{a\geq c_{0}v^{\frac{2}{3}}\}$. Moreover, the evolution generated by (\ref{strongfusioneq}) (or by (\ref{1 over epsilon strongfusioneq})) has the property that it preserves the set of measures supported in this region.

We assume $r(a,v)$ behaves like a power law of $a$ and $v$. For the coagulation kernel $K$, we assume that it has a weak dependence on the surface area of the interacting particles, but it can have a power law behavior in the volume of the coalescing particles.

Since collision does not change if we permute the colliding particles, i.e. $(a,v)\leftrightarrow(a',v')$, the coagulation kernel must satisfy the following symmetry property:
\begin{align}
    K(a,v,a',v') & =K(a',v',a,v),\label{kersym1}
\end{align}
for all $(a,v,a',v')\in(0,\infty)^{4}$.

Concerning the fusion kernel $r$, we assume that $r\in\textup{C}^{1}(\mathbb{R}_{>0}^{2})$ and that there exist constants $R_{0},R_{1}>0$ such that:
\begin{align}
   R_{0}a^{\mu}v^{\sigma}\leq r(a,v)\leq R_{1}a^{\mu}v^{\sigma}, \label{fusion_form}
\end{align}
for all $(a,v)\in(0,\infty)^{2}$ and some coefficients $\mu,\sigma\in\mathbb{R}$.

In order to control the mass of the solutions when $|a-c_{0}v^{\frac{2}{3}}|>0$ in  (\ref{1 over epsilon strongfusioneq}) if $\Lambda$ is small, we require the following technical assumptions on the fusion kernel $r$:
\begin{align}\label{ode_fusion}
\begin{cases}
\partial_{a}r(a,v)-\mu a^{-1} r(a,v)\geq 0,\textup{ and } \partial_{a}r(a,v)\leq B a^{-1}r(a,v),& \textup{ if } \mu>0;  \\
\partial_{a}r(a,v)(a-c_{0}v^{\frac{2}{3}})+r(a,v)\geq 0, \textup{ and } \partial_{a}r(a,v)\leq B a^{-1}r(a,v),&\textup{ if } \mu\leq 0,
\end{cases}
\end{align}
for all $(a,v)\in(0,\infty)^{2},$ with $a\geq c_{0}v^{\frac{2}{3}}$, and for some constant $B>0$. A particular case used in applications that satisfies the above mentioned properties is when $r(a,v)=Ra^{\mu}v^{\sigma}$, with $\mu\geq -1$ and $R\in[R_{0},R_{1}]$. The condition (\ref{ode_fusion}) is not optimal and it would be possible to impose weaker conditions on the fusion kernel. However, this would imply more involved arguments in the proofs later on. We impose the stronger condition (\ref{ode_fusion}) as our main goal is that the statements of our theorems hold for fusion kernels that behave as power laws, case which is included in condition (\ref{ode_fusion}).
\

In comparison to \cite{cristian2022coagulation}, since in this paper we are not interested in the long-time behavior of solutions, we do not assume homogeneity of neither the fusion nor coagulation kernel and we simply assume that they behave like power laws, see (\ref{fusion_form}) and (\ref{lower_bound_kernel}).
In \cite{cristian2022coagulation}, we restricted the analysis to coagulation and fusion kernels which rescale in a similar manner when the size of the particles is changed without modifying their geometry. This was needed for the study of the long-time behavior of particles as the fusion term was chosen in a desire for particles to form a spherical shape in time.  In particular, it meant that, if the particle volume is scaled by a factor $\lambda$, then the diameter is scaled with a factor $\lambda^{\frac{1}{3}}$ and the area scales like $\lambda^{\frac{2}{3}}$ and that we needed to impose the additional assumption that $\frac{2}{3}\mu+\sigma=\gamma-1$, where $\gamma$ is the homogeneity of the coagulation kernel and $\mu,\sigma$ are as in (\ref{fusion_form}). In this work, we treat the more general case of arbitrary $\mu, \sigma\in\mathbb{R}$. In order to deal with this case, we will need to obtain additional moment estimates.

Our main goal for this paper is to prove that all solutions of equation (\ref{1 over epsilon strongfusioneq}) which satisfy some very general moment estimates concentrate their mass around the isoperimetric line $\{a=c_{0}v^{\frac{2}{3}}\}$ as $\Lambda\rightarrow 0$ and tend in an appropriate sense to a measure which can be computed by solving a suitable one-dimensional coagulation equation. Moreover, we prove that solutions of (\ref{1 over epsilon strongfusioneq}) satisfying these moment estimates do exist. The limit measure $f$ of the sequence $\{f_{\Lambda}\}$ acts like a Dirac-like measure in the area variable, namely $f(v,t)=\delta(a-c_{0}v^{\frac{2}{3}})F(v,t)$, with $F$ satisfying the standard one-dimensional coagulation equation. We can then use the known results for the one-dimensional coagulation equations (for example, it was proven mathematically in \cite[Proposition 10.2.1]{book} for solutions in $\textup{C}([0,\infty);\textup{L}^{1}(\mathbb{R}_{>0}))$) to prove that the average volume of the particles increases in time. 

The reason we can reduce the evolution equation for the two-dimensional system to a one-dimensional one is that, as $\Lambda\rightarrow 0$, the fusion process takes place much faster than the collision process and then the particles are transported close to the isoperimetric line $\{a=c_{0}v^{\frac{2}{3}}\}$ almost instantaneously. Equivalently, fusion happens immediately after collision.

On the other hand, if we let $\Lambda\rightarrow \infty$ in equation (\ref{1 over epsilon strongfusioneq}), we recover a two-dimensional coagulation model, in which particles attach to each other at a contact point, forming a ramified-like system in time.  A physical interpretation of this is that, as $\Lambda\rightarrow \infty$, the effect of the fusion term becomes negligible (see Figure \ref{fig2}).

\begin{figure}[H]%
\centering
\captionsetup{justification=centering}
\includegraphics[width=0.5\textwidth, height=5cm]{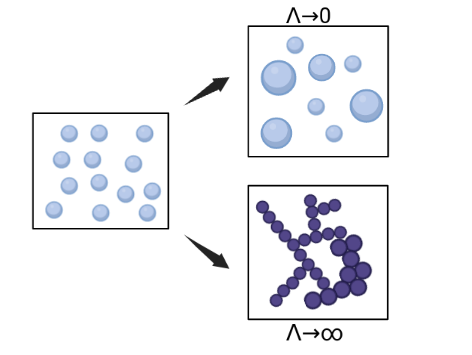}
\caption{System of particles under different fusion times}
\label{fig2}
\end{figure}

The evolution of a system of coagulation equations which can be described by area and volume and where the particles undergo a fusion process after they come in contact has been described in \cite[Chapter 12]{book1}. The specific problem under consideration was the study of aerosol flame reactors. A heuristic analysis of the shapes for the resulting particles for different values of the ratio between the average fusion time and the average collision time can be found in there. This ratio is given by the parameter $\Lambda$ in our model in (\ref{1 over epsilon strongfusioneq}). In the types of models considered in \cite[Chapter 12]{book1}, it is seen that, for small particles or high temperatures, the parameter $\Lambda$ is small. On the contrary, for sufficiently large particles or after the gas has been cooled, we must assume that $\Lambda$ is very large. The results in this paper provide a precise mathematical formulation of the behavior that has been suggested in \cite[Chapter 12]{book1}. The results in this work are complementary to those in \cite{cristian2022coagulation}, in which we consider a particular form of the fusion term for which $\Lambda\approx 1$ for arbitrary times.

Coagulation equations for particle distributions characterized by a single variable have been extensively studied. In particular, the long-time behavior for coagulation equations for which solutions can be explicitly computed has been studied in  \cite{Menon_2004}. The existence of self-similar solutions for general classes of kernels has been obtained in \cite{ESCOBEDO,articlefournier}. 

Multi-dimensional coagulation equations have not been studied as much in the mathematical literature as their one-dimensional counterpart. Several discrete multi-component coagulation problems which are relevant in aerosol physics have been mentioned in \cite{wattis1}. A discrete version of the model in (\ref{strongfusioneq}) has been studied in \cite{wattis2}. The model considered in there includes coagulation of particles and an effect similar to the fusion of particles in (\ref{strongfusioneq}), which has been termed compaction. The diameter of the particles is restricted by the total number of monomers as well as by the isoperimetric inequality. The coagulation and the fusion rates are assumed to be constant. Due to this, the model considered in \cite{wattis2} is explicitly solvable using generating functions. The long-time behavior of the solutions which depends on the ratio between the fusion and coagulation kernels has been then analysed using the explicit formulas of the solutions. 

In \cite{multicom1, multicom2,multicomponent3}, the mathematical properties of some classes of coagulation equations describing clusters that are composed of several types of monomers with different chemical composition are analysed. More recently, uniqueness of the solutions for the models of multi-component coagulation equations considered in \cite{multicom1, multicom2,multicomponent3} has been studied in \cite{throm}.

 More precisely, it has been proven in \cite{multicom1, multicomponent3} that time-dependent solutions for the multi-dimensional coagulation equation concentrate along a line in the space of cluster concentrations for long times for coagulation kernels for which the scaling properties of each of the components are the same for all the species that compose the system. However, as the surface area and volume appear in a less symmetric manner in our model, it does not seem feasible to adapt the proof in \cite{multicom1, multicomponent3} to obtain our result, even in the absence of the fusion term.
 
 Another difference between our model and the one in  \cite{multicom1, multicomponent3} is that the proof in the latter relies on the conservation of mass for each of the types of monomers. Due to the fusion term, we do not have two conserved quantities for (\ref{1 over epsilon strongfusioneq}), but only the volume is conserved. In addition, the fusion term in (\ref{strongfusioneq}), (\ref{1 over epsilon strongfusioneq}) yields a non-trivial evolution of the distribution of particles. The solutions in \cite{multicom1, multicom2, multicomponent3} concentrate along a line with the orientation fixed by the initial distribution of cluster compositions or the source term. Thus, the solutions in 
\cite{multicom1, multicom2, multicomponent3} can concentrate along different lines depending on the initial distribution of particles. On the contrary, in one of the situations considered in this paper, the solutions concentrate always near the isoperimetric line, independently of the initial data. When $\Lambda\rightarrow 0$ in (\ref{1 over epsilon strongfusioneq}), we have that the coagulation operator transports particles away from the isoperimetric line but these are transported extremely fast towards the isoperimetric region due to the fusion term.

\subsection{Notations and plan of the paper}

For $I\subset [0,\infty)^{2}$, we denote by $\textup{C}_{\textup{c}}(I)$ and $\textup{C}_{\textup{0}}(I)$ the space of continuous functions on $I$ with compact support and the space of continuous functions on $I$ which vanish at infinity, respectively, both endowed with the supremum norm. $\mathscr{M}_{+}(I)$ will denote the space of non-negative Radon measures, while $\mathscr{M}_{+,\text{b}}(I)$ will be the space of non-negative, bounded Radon measures, which we endow with the weak-$^{\ast}$ topology.
\

We make in addition the following simplifications:
\begin{itemize}
\item We use the notation $\eta:=(a,v)$. We will use interchangeably both notations for convenience. 
    \item We will use the notation $f(a,v)\der v \der a$ or $f(\eta)\der \eta$ for Radon measures, independently of the fact the measure may not be absolutely continuous with respect to the Lebesgue measure. 
    \item $M_{k,l}(f):=\int_{(0,\infty)^{2}}a^{k}v^{l}f(a, v)\der v \der a$, for some $k,l\in\mathbb{R}.$
    \item For a suitably chosen $\varphi:\mathbb{R}^{2}_{>0}\rightarrow \mathbb{R}$ and for $(a,v,a',v')\in(0,\infty)^{4}$, we will denote:
   \begin{align}
   \chi_{\varphi}(a,v,a',v')& :=\varphi(a+a',v+v')-\varphi(a,v)-\varphi(a',v'); \nonumber \\
   \langle \mathbb{K}[f],\varphi\rangle& :=\frac{1}{2}\int_{(0,\infty)^{2}}\int_{(0,\infty)^{2}} K(\eta,\eta')\chi_{\varphi} (\eta,\eta')f(\eta')f(\eta) \der \eta'\der \eta. \nonumber
   \end{align}
   \item We use $C$ to denote a generic constant which may differ from line to line and depends only on the parameters characterizing the kernels $K$ and $r$.
      \item We use the symbols $\lesssim$ and $\gtrsim$ when the inequalities hold up to a constant, i.e. $f\lesssim g$ if and only if $f\leq C g$, for some $C>0$.
\end{itemize}

The structure of the paper is as follows. In the rest of this section, we establish the setting and state the main definitions and results.

In Section \ref{epsilon goes to zero}, we prove that there exists a limit for the sequence of solutions of equation (\ref{1 over epsilon strongfusioneq}) as $\Lambda\rightarrow 0$. To this end, we first prove that the mass of solutions concentrates around the isoperimetric line. This is done by looking at the adjoint equation of (\ref{1 over epsilon strongfusioneq}). The fact that the measures take small values if we are at a positive distance from the line $\{a=c_{0}v^{\frac{2}{3}}\}$ together with the fact that we can control large values of the area $a$ suffices to prove the equicontinuity in time of solutions and conclude that a limit of the sequence exists.
We then prove that the found limit is a solution for the standard one-dimensional coagulation equation. This is since now the $a$ variable acts like $c_{0}v^{\frac{2}{3}}$ and we can omit the fusion term by testing (\ref{1 over epsilon strongfusioneq}) with functions only depending on the $v$ variable.

In Section \ref{the case n goes to infinity}, we deal with the case when $\Lambda\rightarrow\infty$ in equation (\ref{1 over epsilon strongfusioneq}). We prove that a limit exists as $\Lambda\rightarrow\infty$ and that the limit satisfies a two-dimensional coagulation equation where the interaction of particles consists of particles which attach at a contact point. The proof of this result is straightforward after obtaining suitable moment estimates for the solutions, which are independent of the value of $\Lambda$. 
\subsection{Setting and main results} \label{setting}

We work with non-negative continuous kernels on $(0,\infty)^{4}$ that, in addition to the properties already stated, i.e. (\ref{kersym1}), have the following bounds:
\begin{align}
K_{1}(v^{-\alpha}v'^{\beta}+v'^{-\alpha}v^{\beta})\leq K(a,v,a',v')\leq K_{0}(v^{-\alpha}v'^{\beta}+v'^{-\alpha}v^{\beta}), \label{lower_bound_kernel}
\end{align}
for some $K_{1},K_{0}>0$, for all $a,v,a',v'$ and for the following coefficients:
\begin{align}
  &\alpha>0 \textup{ and } \beta\in (0,1) \textup{ such that } \beta-\alpha\in(0,1). \label{alpha non neg}
  \end{align}

Notice that condition (\ref{lower_bound_kernel}) implies that the kernel has a weak dependence on the area variable, but $K$ is not necessarily independent of the area variable.

\
Since we work with physically relevant particles, i.e. the particles for which the isoperimetric inequality is satisfied, it is helpful to define the following space
\begin{align}
   \mathscr{M}^{I}_{+}(\mathbb{R}_{>0}^{2}):=\{h\in \mathscr{M}_{+}(\mathbb{R}_{>0}^{2})\text{  } | \text{  } h(\{a<c_{0}v^{\frac{2}{3}}\})=0\}. \label{isoperimetricinitial}
\end{align}
The superscript $I$ stands for isoperimetric. We endow the newly-defined space with the weak-$^{\ast}$ topology on $\mathscr{M}_{+}(\mathbb{R}_{>0}^{2})$. Similarly, we denote
\begin{align}
   \mathscr{M}^{I}_{+,\textup{b}}(\mathbb{R}_{>0}^{2}):=\{h\in \mathscr{M}_{+,\textup{b}}(\mathbb{R}_{>0}^{2})\text{  } | \text{  } h(\{a<c_{0}v^{\frac{2}{3}}\})=0\}.
\end{align}
\begin{defi}\label{definitiontimedependent}
Fix $\Lambda>0$. Let $K:(0,\infty)^{4}\rightarrow [0,\infty)$ be a continuous kernel satisfying (\ref{kersym1}), 
(\ref{lower_bound_kernel}) and (\ref{alpha non neg}). Assume the fusion kernel $r\in\textup{C}^{1}(\mathbb{R}_{>0}^{2})$ satisfies (\ref{fusion_form}) and (\ref{ode_fusion}). Let $f_{\Lambda}\in\textup{C}([0,\infty);\mathscr{M}^{I}_{+}(\mathbb{R}_{>0}^{2}))$. We say that $f_{\Lambda}$ is a solution for the weak version of the time-dependent $\Lambda$-fusion problem if, for every $T>0$,
\begin{align*}
 \sup_{t\in[0,T]}  \int_{(0,\infty)^{2}}(v^{-\alpha}+v^{\beta})f_{\Lambda}(a,v,t)\der v\der a<\infty
    \end{align*}
and, for all $\varphi\in\textup{C}^{1}_{\textup{c}}([0,\infty);\compactfundif)$ and $t\in[0,\infty)$
\begin{align} \label{weak_form_time_dependent}
&\int_{(0,\infty)^{2}}f_{\Lambda}(\eta,t)\varphi(\eta,t)\der \eta-\int_{(0,\infty)^{2}}f_{\textup{in}}(\eta)\varphi(\eta,0)\der \eta-\int_{0}^{t}\int_{(0,\infty)^{2}}f_{\Lambda}(\eta,s)\partial_{s}\varphi(\eta,s)\der \eta\der s\nonumber\\
&=\int_{0}^{t}\langle \mathbb{K}[f_{\Lambda}(s)],\varphi(s)\rangle\der s+\frac{1}{\Lambda}\int_{0}^{t}\int_{(0,\infty)^{2}}r(\eta)(c_{0}v^{\frac{2}{3}}-a)f_{\Lambda}(\eta,s)\partial_{a}\varphi(\eta,s)\der \eta\der s.
\end{align}
\end{defi}
\begin{rmk}
Functions $f_{\Lambda}$ as in Definition \ref{definitiontimedependent} exist and the methods to prove their existence are similar to the ones used to prove existence of self-similar solutions in \cite{cristian2022coagulation} (in order to derive some moment estimates in \cite{cristian2022coagulation}, ideas from the one-dimensional case in \cite{dust, ESCOBEDO} were adapted). A sketch for proving their existence will be shown in Proposition \ref{existence of solutions for the equation}.
\end{rmk}
\subsection*{The case of fast fusion}
\begin{rmk}
From now on, in order to simplify the notation, we replace $\Lambda$ by $\epsilon$ in (\ref{1 over epsilon strongfusioneq}) when we consider the case $\Lambda\rightarrow 0$ and  we replace $\Lambda$ by $\frac{1}{\epsilon}$ in (\ref{1 over epsilon strongfusioneq}) when we consider the case $\Lambda\rightarrow \infty$.
\end{rmk}
\begin{teo}\label{existencelimitandcorrectsupport}
Let $K:(0,\infty)^{4}\rightarrow [0,\infty)$ be a continuous kernel satisfying (\ref{kersym1}), 
(\ref{lower_bound_kernel}) and (\ref{alpha non neg}). Assume the fusion kernel $r\in\textup{C}^{1}(\mathbb{R}_{>0}^{2})$ satisfies (\ref{fusion_form}) and (\ref{ode_fusion}) with $\mu>0$. Assume in addition that $\int_{(0,\infty)^{2}}(v^{-\mu-3}+v^{\mu+3}+v^{\sigma(\mu+3)}+a^{\mu+3})f_{\textup{in}}(a,v)\der v\der a<\infty$. Let $T>0$. Then we can construct $f_{\epsilon}$ as in Definition \ref{definitiontimedependent}, for every $\epsilon\in(0,1)$. For this sequence, we have that there exists a constant $C(T)>0$, which is independent of $\epsilon\in(0,1)$, such that
\begin{align}\label{nicemomentestimates}
   \sup_{t\in[0,T]} \int_{(0,\infty)^{2}}(v^{-\mu-3}+v^{\mu+3}+v^{\sigma(\mu+3)}+a^{\mu+3})f_{\epsilon}(a,v,t)\der v\der a\leq C(T).
\end{align}
Moreover, there exists a subsequence (which we do not relabel) and $\overline{f}\in\textup{C}((0,T];  \mathscr{M}^{I}_{+}(\mathbb{R}_{>0}^{2}))$ such that $f_{\epsilon}(t)\rightarrow\overline{f}(t)$ as $\epsilon\rightarrow 0$ in the sense of measures, for every $t\in[\overline{\sigma},T]$ and every $\overline{\sigma}>0$.
\end{teo}
\begin{rmk}
Theorem \ref{existencelimitandcorrectsupport} holds true also in the case $\mu\leq 0$ if we assume instead that $\int_{(0,\infty)^{2}}(v^{-1}+v^{2}+a)f_{\textup{in}}(a,v)\der v\der a<\infty$ (plus some additional moment bound of the form $M_{0,d}$, with $d$ depending on $\sigma$, which does not offer much qualitative information), which in turn will imply that there exists a constant $C(T)>0$, which is independent of $\epsilon\in(0,1)$, such that
\begin{align*}
   \sup_{t\in[0,T]} \int_{(0,\infty)^{2}}(v^{-1}+v^{2}+a)f_{\epsilon}(a,v,t)\der v\der a\leq C(T).
\end{align*}
The methods to prove the two cases, $\mu>0$ and $\mu\leq 0$, are similar up to minor technicalities and thus we restrict our attention to the case $\mu>0$ for simplicity of notation.
\end{rmk}
\begin{rmk}
Theorem \ref{existencelimitandcorrectsupport} says that we can construct a sequence of functions $\{f_{\epsilon}\}_{\epsilon\in(0,1)}$ such that a limit exists. However,  Theorem \ref{existencelimitandcorrectsupport} (and all the results stated below in this subsection) is valid for any sequence of functions $\{f_{\epsilon}\}_{\epsilon\in(0,1)}$ for which (\ref{nicemomentestimates}) holds. Note that uniqueness of coagulation equations is, with the exception of  some particular choices of coagulation kernels, still an open problem. Condition (\ref{nicemomentestimates}) is however a rather strong condition and an upper bound for the moments for the initial condition is a sufficient condition to construct a sequence for which (\ref{nicemomentestimates}) is true. We expect (\ref{nicemomentestimates}) to hold under weaker assumptions than the ones stated. This will be considered in a future work.

\end{rmk}
\begin{lem}\label{lemma dirac measure}
  Let $K:(0,\infty)^{4}\rightarrow [0,\infty)$ be a continuous kernel satisfying (\ref{kersym1}), 
(\ref{lower_bound_kernel}) and (\ref{alpha non neg}). Assume the fusion kernel $r\in\textup{C}^{1}(\mathbb{R}_{>0}^{2})$ satisfies (\ref{fusion_form}) and (\ref{ode_fusion}). Let $T>0$.  Let $f_{\epsilon}$ and $\overline{f}$ as in Theorem \ref{existencelimitandcorrectsupport}. Then $\overline{f}$ has the form 
\begin{align}\label{form of limit}
    \overline{f}(a,v,t)=F(v,t)\delta(a-c_{0}v^{\frac{2}{3}})
\end{align}
in the sense of measures. More precisely, there exists $F\in\textup{C}((0,T];\mathscr{M}_{+}(\mathbb{R}_{>0}))$ such that
\begin{align*}
\int_{(0,\infty)^{2}}\overline{f}(a,v,t)\varphi(a,v)\der v \der a=\int_{(0,\infty)^{2}}F(v,t)\varphi(a,v)\delta(a-c_{0}v^{\frac{2}{3}})\der a\der v,
\end{align*}
for every $\varphi\in\textup{C}_{0}(\mathbb{R}_{>0}^{2})$.
\end{lem}

\begin{defi}\label{definitionlimit}
   Let $F$ be as in (\ref{form of limit}) and $f_{\textup{in}}$ as in Theorem \ref{existencelimitandcorrectsupport}. We define the following initial value
 \begin{align}\label{initial datum}
\int_{(0,\infty)}F(v,0)\varphi(v)\der v:=\int_{(0,\infty)^{2}}f_{\textup{in}}(a,v)\varphi(v)\der a\der v,
\end{align}
for every $\varphi\in\textup{C}_{0}(\mathbb{R}_{>0})$.
\begin{defi}\label{limit function}
Let $\epsilon\in(0,1)$, fix $T>0$ and let $f_{\epsilon}$ be as in Theorem \ref{existencelimitandcorrectsupport}. Define $F_{\epsilon}\in\textup{C}([0,T];\mathscr{M}_{+}(\mathbb{R}_{>0}))$ as
\begin{align}\label{one dim measures}
\int_{(0,\infty)}F_{\epsilon}(v,t)\varphi(v)\der v:=\int_{(0,\infty)^{2}}f_{\epsilon}(a,v,t)\chi_{\epsilon}(a)\varphi(v)\der v \der a,
\end{align}
for every $t\in[0,T]$ and $\varphi\in\textup{C}_{0}(\mathbb{R}_{>0})$ and where $\chi_{\epsilon}:(0,\infty)\rightarrow[0,1]$ is a continuous function such that $\chi(a)=1$ on $(0,\frac{1}{\epsilon^{2}}]$ and $\chi(a)=0$ on $[\frac{2}{\epsilon^{2}},\infty)$.
\end{defi}
\end{defi}
\begin{teo}\label{pass back to standard coagulation}
 Let $\{F_{\epsilon}\}_{\epsilon\in(0,1)}$ be as in Definition \ref{limit function} and $F$ as in (\ref{form of limit}) with initial value as in Definition \ref{definitionlimit}. We then have that 
\begin{align*}
\int_{(0,\infty)}F_{\epsilon}(v,t)\varphi(v)\der v\rightarrow \int_{(0,\infty)}F(v,t)\varphi(v)\der v,
\end{align*}
as $\epsilon\rightarrow 0$, for every $t\in[0,T]$ and every $\varphi\in\textup{C}_{0}(\mathbb{R}_{>0})$. Moreover, $F\in\textup{C}([0,T];\mathscr{M}_{+}(\mathbb{R}_{>0}))$ and $F$ satisfies the standard one-dimensional coagulation equation, namely, for every $t\in[0,T]$ and every $\varphi\in\textup{C}_{\textup{c}}(\mathbb{R}_{>0})$, the following holds
\begin{align*}
&\int_{(0,\infty)}F(v,t)\varphi(v)\der v-\int_{(0,\infty)}F(v,0)\varphi(v)\der v\nonumber\\
&=\int_{0}^{t}\int_{(0,\infty)}\int_{(0,\infty)}K(c_{0}v^{\frac{2}{3}},v,c_{0}v'^{\frac{2}{3}},v')F(v,s)F(v',s)[\varphi(v+v')-\varphi(v)-\varphi(v')]\der v'\der v\der s.
\end{align*}
\end{teo}
\subsection*{The case of negligible fusion}
\begin{rmk}
In order to simplify the notation, in the case when $\Lambda\rightarrow \infty$, we replace $\Lambda$ by $\frac{1}{\epsilon}$, for $\epsilon>0.$ Thus, for $T>0$ and $t\in[0,T]$, we look at the equation
\begin{align}\label{epsilonic_convergence n case}
&\int_{(0,\infty)^{2}}f_{\epsilon}(\eta,t)\varphi(\eta,t)\der \eta-\int_{(0,\infty)^{2}}f_{\textup{in}}(\eta)\varphi(\eta,0)\der \eta-\int_{0}^{t}\int_{(0,\infty)^{2}}f_{\epsilon}(\eta,s)\partial_{s}\varphi(\eta,s)\der \eta\der s\nonumber\\
&=\int_{0}^{t}\langle \mathbb{K}[f_{\epsilon}(s)],\varphi(s)\rangle\der s+\epsilon\int_{0}^{t}\int_{(0,\infty)^{2}}r(\eta)(c_{0}v^{\frac{2}{3}}-a)f_{\epsilon}(\eta,s)\partial_{a}\varphi(\eta,s)\der \eta\der s,
\end{align}
for $\epsilon\in(0,1)$, $\varphi\in\textup{C}^{1}_{0}(\mathbb{R}_{>0}^{2})$ and with $f_{\textup{in}}\in\mathscr{M}^{I}_{+}(\mathbb{R}_{>0}^{2})$.
\end{rmk}
\begin{teo}\label{case of slow fusion}
 Let $K:(0,\infty)^{4}\rightarrow [0,\infty)$ be a continuous kernel satisfying (\ref{kersym1}), 
(\ref{lower_bound_kernel}) and (\ref{alpha non neg}). Assume the fusion kernel $r\in\textup{C}^{1}(\mathbb{R}_{>0}^{2})$ satisfies (\ref{fusion_form}) and (\ref{ode_fusion}) with $\mu>0$. Assume in addition that $\int_{(0,\infty)^{2}}(v^{-\mu-2}+v^{\mu+2}+v^{\sigma(\mu+2)}+a^{\mu+2})f_{\textup{in}}(a,v)\der v\der a<\infty$. Let $T>0$. Then we can construct $f_{\epsilon}$ as in Definition \ref{definitiontimedependent} satisfying equation (\ref{epsilonic_convergence n case}), for every $\epsilon\in(0,1)$. For this sequence, we have that there exists a constant $C(T)>0$, which is independent of $\epsilon\in(0,1)$, such that
\begin{align}\label{nicemomentestimates --}
   \sup_{t\in[0,T]} \int_{(0,\infty)^{2}}(v^{-\mu-2}+v^{\mu+2}+v^{\sigma(\mu+2)}+a^{\mu+2})f_{\epsilon}(a,v,t)\der v\der a\leq C(T)
\end{align}
and that there exists a subsequence (which we do not relabel) and $\underline{f}\in\textup{C}([0,T];  \mathscr{M}^{I}_{+}(\mathbb{R}_{>0}^{2}))$ such that
\begin{align*}
\int_{(0,\infty)^{2}}f_{\epsilon}(\eta,t)\varphi(\eta)\der \eta\rightarrow \int_{(0,\infty)^{2}}\underline{f}(\eta,t)\varphi(\eta)\der \eta,
\end{align*}
as $\epsilon\rightarrow 0$, for every $t\in[0,T]$ and every $\varphi\in\textup{C}_{0}(\mathbb{R}^{2}_{>0})$. Additionally, we have that $\underline{f}$ satisfies a standard two-dimensional coagulation equation, namely, for every $t\in[0,T]$ and every $\varphi\in\textup{C}_{\textup{c}}(\mathbb{R}_{>0}^{2})$, the following holds
\begin{align}\label{two dim final}
&\int_{(0,\infty)^{2}}\underline{f}(\eta,t)\varphi(\eta)\der \eta-\int_{(0,\infty)^{2}}f_{\textup{in}}(\eta)\varphi(\eta)\der \eta\nonumber\\
&=\int_{0}^{t}\int_{(0,\infty)^{2}}\int_{(0,\infty)^{2}}K(\eta,\eta')\underline{f}(\eta,s)\underline{f}(\eta',s)[\varphi(\eta+\eta')-\varphi(\eta)-\varphi(\eta')]\der \eta'\der \eta\der s.
\end{align}
\end{teo}
\begin{rmk}\label{remark conservation mass}
The function $\underline{f
}$ found in Theorem \ref{case of slow fusion} satisfies
\begin{align*}
    \int_{(0,\infty)^{2}}a\underline{f}(\eta,t)\der \eta =  \int_{(0,\infty)^{2}}af_{\textup{in}}(\eta)\der \eta \textup{ and }   \int_{(0,\infty)^{2}}v\underline{f}(\eta,t)\der \eta =  \int_{(0,\infty)^{2}}vf_{\textup{in}}(\eta)\der \eta
\end{align*}
for every $t\in[0,T].$
\end{rmk}
\section{The case of fast fusion}\label{epsilon goes to zero}
\subsection{Existence of a limit of solutions of coagulation equations with fast fusion}\label{existence of limit section}
Let $T>0$ and $t\in[0,T]$. We look at the equation
\begin{align}\label{epsilonic_convergence}
&\int_{(0,\infty)^{2}}f_{\epsilon}(\eta,t)\varphi(\eta,t)\der \eta-\int_{(0,\infty)^{2}}f_{\textup{in}}(\eta)\varphi(\eta,0)\der \eta-\int_{0}^{t}\int_{(0,\infty)^{2}}f_{\epsilon}(\eta,s)\partial_{s}\varphi(\eta,s)\der \eta\der s\nonumber\\
&=\int_{0}^{t}\langle \mathbb{K}[f_{\epsilon}(s)],\varphi(s)\rangle\der s+\frac{1}{\epsilon}\int_{0}^{t}\int_{(0,\infty)^{2}}r(\eta)(c_{0}v^{\frac{2}{3}}-a)f_{\epsilon}(\eta,s)\partial_{a}\varphi(\eta,s)\der \eta\der s,
\end{align}
for $\epsilon\in(0,1)$, $\varphi\in\textup{C}^{1}_{0}(\mathbb{R}_{>0}^{2})$ and with $f_{\textup{in}}\in\mathscr{M}^{I}_{+}(\mathbb{R}_{>0}^{2})$.

We begin by remembering the truncated functions used to prove the existence of solutions for truncated versions of coagulation equations allowing fusion of particles, which is done using a fixed point argument. More details can be found in \cite{cristian2022coagulation}.

We define $K_{R}:(0,\infty)^{4}\rightarrow[0,\infty)$ to be a continuous function such that:
 \begin{align}
     K_{R}(a,v,a',v')=\min\{K(a,v,a',v'), R\}, \label{truncationkernel}
 \end{align}
 where $K$ satisfies the upper bound in (\ref{lower_bound_kernel}) and take $\xi_{R}:\mathbb{R}_{>0}\rightarrow [0,\infty)$ to be continuous and defined in the following manner:
 \begin{align}
     \xi_{R}(v) & =0, & &\text{ when } v\geq 2R, \\
     \xi_{R}(v)& =1, & &\text { on } (0,R].
 \end{align}
Then, for $\varphi\in\textup{C}^{1}_{0}(\mathbb{R}_{>0}^{2})$, we denote by
\begin{align}\label{kernel_term}
    \langle \mathbb{K}_{R}[f],\varphi\rangle:=\frac{1}{2}\int_{(0,\infty)^{2}}\int_{(0,\infty)^{2}}K_{R}(\eta,\eta')\xi_{R}(v+v')[\varphi(\eta+\eta')-\varphi(\eta)-\varphi(\eta')]f(\der\eta)f(\der\eta').
\end{align}
For the fusion term, we use the following truncation:
\begin{align}\label{cutfusion}
r_{\delta}(a,v):= \frac{r( \eta)\max\{v^{\sigma}, L\delta\}}{v^{\sigma}(1+\delta a^{\mu})},
\end{align}
for $\delta\in(0,1)$ and some fixed $L>0$. $L$ was chosen in \cite{cristian2022coagulation}  to be
$  L:=\frac{12}{R_{0}(1-\gamma)}$,
where $R_{0}$ is as in (\ref{fusion_form}), in order to obtain existence of self-similar profiles for equation (\ref{strongfusioneq}). 

For functions $f$ satisfying $f\in\textup{C}^{1}([0,\infty);\mathscr{M}^{I}_{+}(\mathbb{R}_{>0}^{2}))$ with
\begin{align}\label{compact support truncated functions}
f\Big(\mathbb{R}^{2}_{>0}\setminus[c_{0}\overline{\epsilon}^{\frac{2}{3}},\infty)\times[\overline{\epsilon},2R),t\Big)=0
\end{align} and such that
\begin{align}\label{finalization_space}
\sup_{t\in[0,T]}\int_{(0,\infty)^{2}}(1+a)f_{\overline{\epsilon},R,\delta}(a,v,t)\der v \der a<\infty,
\end{align}
for all times $T\in[0,\infty),$ we define the space
\begin{align}\label{existencespace for f}
U_{\overline{\epsilon},R}&:= \{f\in\textup{C}^{1}([0,\infty);\mathscr{M}^{I}_{+}(\mathbb{R}_{>0}^{2})), f \textup{ satisfies } (\ref{compact support truncated functions})\textup{ and } (\ref{finalization_space}) \}.
\end{align}

\begin{prop}\label{existence of solutions for the equation}
Assume $f_{\textup{in}}\in\mathscr{M}_{+}^{I}(\mathbb{R}_{>0}^{2})$ and
\begin{align}\label{condition on initial solution for existence of solutions}
    \int_{(0,\infty)^{2}}(v^{-1}+v^{2}+a)f_{\textup{in}}(\eta)\der \eta <\infty.
\end{align}Let $T>0$ and fix $\epsilon\in(0,1)$. Then there exists an $f_{\epsilon}\in\textup{C}([0,T];\mathscr{M}^{I}_{+}(\mathbb{R}^{2}_{>0}))$ as in Definition \ref{definitiontimedependent} with
\begin{align*}
\sup_{t\in[0,T]}\int_{(0,\infty)^{2}}(v^{-1}+v^{2}+a)f_{\epsilon}(\eta,t)\der \eta \leq C(T).
\end{align*}
\end{prop}
\begin{rmk}
In order to prove the next proposition, we will need that
\begin{align}\label{higher order moments}
   \sup_{t\in[0,T]} \int_{(0,\infty)^{2}}a^{\mu+3}f_{\epsilon}(\der\eta,t)\leq C(T).
\end{align}
While the estimates in (\ref{condition on initial solution for existence of solutions}) suffice for the existence of solutions as in Definition \ref{definitiontimedependent}, in order to obtain an upper bound for moments involving higher powers of the area, we need the additional assumption that
\begin{align*}
   \sup_{t\in[0,T]} \int_{(0,\infty)^{2}}(v^{-\mu-3}+v^{\mu+3})f_{\epsilon}(\der\eta,t)\leq C(T).
\end{align*}
The proof of (\ref{higher order moments}) relies on the fact that the terms of the form $a^{\mu+2}v^{\beta}$, which appear due to the form of the coagulation kernel, can be bounded by
\begin{align*}
a^{\mu+2}v^{\beta}\lesssim a^{\mu+3}+v^{\beta(\mu+3)}.
\end{align*}
For more details, we refer to \cite[Subsection 3.3]{cristian2022coagulation}.
\end{rmk}
\begin{prop}\label{solutions stay close to the isoperimetric line}
Let $T>0$. Suppose $\{f_{\epsilon}\}_{\epsilon\in(0,1)}$ is the sequence of solutions found in Proposition \ref{existence of solutions for the equation}. Assume $\mu>0$. Assume in addition that 
\begin{align}\label{upper estimates for region positive distance from iso line}
\sup_{t\in[0,T]}\int_{(0,\infty)^{2}} (a^{\mu+3}+v^{\sigma(\mu+3)})f_{\epsilon}(\der\eta,t)\leq C(T),\end{align}
for every $\epsilon\in(0,1)$. For every $\overline{\sigma}>0$ and for every $\delta_{1}, \delta_{2}\in(0,1),$ there exists $\epsilon_{\delta_{1}, \delta_{2}}\in(0,1)$, which is independent of $\overline{\sigma}$, such that
\begin{align}\label{prop small outside iso line}
\bigg|\int_{(0,\infty)^{2}}f_{\epsilon}(a,v,t)\varphi(a,v)\der v\der a\bigg|\leq \delta_{2},
\end{align}
for every $t\geq \overline{\sigma}$, for all $\epsilon\leq \epsilon_{\delta_{1}, \delta_{2}}$ and for all $\varphi\in\textup{C}_{0}(\mathbb{R}^{2}_{>0})$ with $||\varphi||_{\infty}\leq 1$ such that $\varphi(a,v)=0$, when $a\leq c_{0}v^{\frac{2}{3}}+\delta_{1}$. 
\end{prop}
\begin{prop}[Equicontinuity for positive times]\label{existence of limit for f epsilon}
Let $\overline{\sigma}>0$. Suppose $\{f_{\epsilon}\}_{\epsilon\in(0,1)}$ is a sequence of solutions as in Proposition \ref{solutions stay close to the isoperimetric line}. Let $\varphi\in\textup{C}_{0}(\mathbb{R}^{2}_{>0})$ with $||\varphi||_{\infty}\leq 1$. For every $\overline{\epsilon}\in(0,1)$, there exists $\overline{\delta}\in(0,\overline{\sigma})$ such that 
\begin{align*}
\big|\int_{(0,\infty)^{2}}\varphi(\eta)f_{\epsilon}(\der \eta,t)-\int_{(0,\infty)^{2}}\varphi(\eta)f_{\epsilon}(\der \eta,s)\big|\leq \overline{\epsilon},
\end{align*}
for all $t,s\geq\overline{\sigma}$ such that $|t-s|\leq \overline{\delta}$ and for all $\epsilon$ sufficiently small.
\end{prop}
\begin{proof}[Proof of Proposition \ref{existence of solutions for the equation}]
We study the following truncated version of equation (\ref{epsilonic_convergence}):
\begin{align}\label{truncated epislonic convergence}
\partial_{t}\int_{(0,\infty)^{2}}f_{\epsilon,R}(\eta,t)\varphi(\eta)\der \eta=\langle \mathbb{K}_{R}[f_{\epsilon,R}(t)],\varphi\rangle+\frac{1}{\epsilon}\int_{(0,\infty)^{2}}\tilde{r}(\eta)(c_{0}v^{\frac{2}{3}}-a)f_{\epsilon,R}(\eta,t)\partial_{a}\varphi(\eta)\der \eta ,
\end{align}
with
\begin{align*}
\tilde{r}(\eta)=
\begin{cases}
r_{\delta}(\eta), & \mu>0, \\
r(\eta), &\mu\leq 0,
\end{cases}
\end{align*}
where $r_{\delta}$ is as in (\ref{cutfusion}) and $\mathbb{K}_{R}$ as in (\ref{kernel_term}).

We prove existence and uniqueness of $f_{\epsilon,R}$ in the space $U_{\tilde{\epsilon},R}$ defined as in (\ref{existencespace for f}), for some $\tilde{\epsilon}\in(0,1).$  We keep the notation $f_{\epsilon,R}$ for convenience. Notice however that $f_{\epsilon,R}$ depends on $\epsilon,R,\tilde{\epsilon},\delta$ when $\mu>0$ and on $\epsilon,R,\tilde{\epsilon}$ when $\mu\leq 0$.

We look at the system of characteristic equations:
\begin{equation}
\left\{\begin{aligned}
\partial_{t}A(a_{0},v_{0},t)&=\frac{1}{\epsilon}\tilde{r}(A,V)(c_{0}V^{\frac{2}{3}}-A), & A(a_{0},v_{0},0)&=a_{0}; \\
 \partial_{t}V(a_{0},v_{0},t)&=0, & V(a_{0},v_{0},0)&=v_{0};   \\
\partial_{t}c(a_{0},v_{0},t)&=\frac{1}{\epsilon}\partial_{A}[\tilde{r}(A,V)(A-c_{0}V^{\frac{2}{3}})]c(a_{0},v_{0},t),& c(a_{0},v_{0},0)&=1.
   \end{aligned}\right.
   \end{equation}
   \begin{rmk}\label{remark notation ode}
   We remark that $V(a_{0},v_{0},t)\equiv v_{0}.$ Fix $t\geq 0$, we denote the pair $(A(a_{0},v_{0},t),$ $V(a_{0},v_{0},t))$ $=:$ $ \phi_{t,\epsilon}(a_{0},v_{0})$. We fix $v_{0}$, we define $x_{t,v_{0}}^{\epsilon}:\mathbb{R}_{>0}\rightarrow\mathbb{R}$ by $x_{t,v_{0}}^{\epsilon}(a_{0}):=A(a_{0},v_{0},t)$.
\end{rmk}

We first prove the existence and uniqueness of functions $\overline{F}_{\epsilon,R}\in\textup{C}([0,T];\mathscr{M}^{I}_{+,\textup{b}}(\mathbb{R}_{>0}^{2}))$ satisfying
\begin{align*}
   \partial_{t}\int_{\mathbb{R}^{2}_{>0}} & \overline{F}_{\epsilon,R}(\eta,t)\varphi(\eta)\der \eta =\frac{1}{2}\int_{\mathbb{R}^{2}_{>0}}\int_{\mathbb{R}^{2}_{>0}}K_{R}(x_{t,V}^{\epsilon}(A),V,x_{t,V'}^{\epsilon}(A'),V')\xi_{R}(V+V')\overline{F}_{\epsilon,R}(A',V',t)\\
  &\overline{F}_{\epsilon,R}(A,V,t)[\varphi(\phi_{t,\epsilon}^{-1}(\phi_{t,\epsilon}(A,V)+\phi_{t,\epsilon}(A',V')))-\varphi(A,V)-\varphi(A',V')]\der V' \der A' \der V\der A,
\end{align*}
for every $\varphi\in\compactfun$. For this, we repeat the arguments used in \cite[Proposition 3.1]{cristian2022coagulation}. 

We then define $f_{\epsilon,R}\in\textup{C}^{1}([0,\infty);\mathscr{M}^{I}_{+}(\mathbb{R}_{>0}^{2}))$ as
\begin{align}
      \int_{(0,\infty)^{2}}f_{\epsilon,R}(a,v,t)\varphi(a,v)\der \eta=\int_{(0,\infty)^{2}}\overline{F}_{\epsilon,R}(a,v,t)\varphi(\phi_{t,\epsilon}(a,v))\der \eta,
\end{align}
for every $\varphi\in\vanishfun$. Notice that the functions $f_{\epsilon,R}$ defined in this manner will satisfy equation (\ref{truncated epislonic convergence}). For more details, see \cite[Proposition 3.1]{cristian2022coagulation}.

We are now left to prove uniform estimates for $f_{\epsilon,R}$ in order to finish the proof. Due to the choice of the space $U_{\tilde{\epsilon},R},$ we can test (\ref{truncated epislonic convergence}) with $\varphi(a,v)=a$ and $\varphi(a,v)=v^{d},$ with $d\in\mathbb{R}$. When $d\leq 1,$ we obtain:
\begin{align*}
    \int_{(0,\infty)^{2}}(v^{d}+a)f_{\epsilon,R}(\eta,t)\der \eta \leq  \int_{(0,\infty)^{2}}(v^{d}+a)f_{\textup{in}}(\eta)\der \eta.
\end{align*}
For $d=2$, since $M_{0,1-\alpha}(f_{\epsilon,R})$ is now uniformly bounded, we have
\begin{align*}
    \int_{(0,\infty)^{2}}v^{2}f_{\epsilon,R}(\eta,t)\der \eta &\leq  \int_{(0,\infty)^{2}}v^{2}f_{\textup{in}}(\eta)\der \eta+C\int_{0}^{t}M_{0,1-\alpha}(f_{\epsilon,R}(s))M_{0,1+\beta}(f_{\epsilon,R}(s))\der s\\
    &\leq \int_{(0,\infty)^{2}}v^{2}f_{\textup{in}}(\eta)\der \eta+C\int_{0}^{t}M_{0,2}(f_{\epsilon,R}(s))\der s+C.
    \end{align*}
Thus, there exists $C(t)>0$ such that 
\begin{align*}
    \int_{(0,\infty)^{2}}v^{2}f_{\epsilon,R}(\eta,t)\der \eta &\leq C(t) \int_{(0,\infty)^{2}}v^{2}f_{\textup{in}}(\eta)\der \eta.
\end{align*}

\subsubsection*{Equicontinuity}
Fix $\varphi\in\compactfundif$. Let $s,t\in[0,T]$. Assume without loss of generality $s\leq t$. Then
\begin{align}\label{equicontinuity}
    &\bigg|\int_{(0,\infty)^{2}}[f_{\epsilon,R}(\eta,t)-f_{\epsilon,R}(\eta,s)]\varphi(\eta)\der \eta\bigg|\leq\nonumber\\
    &\lesssim ||\varphi||_{\infty}\int_{s}^{t}M_{0,-\alpha}(f_{\epsilon,R}(z))M_{0,\beta}(f_{\epsilon,R}(z))\der z+\frac{1}{\epsilon}\int_{s}^{t}\int_{(0,\infty)^{2}}a^{\mu}v^{\sigma}(a+c_{0}v^{\frac{2}{3}})f_{\epsilon,R}(\eta,z)|\partial_{a}\varphi(\eta)|\der \eta\der z\nonumber\\
    &\leq ||\varphi||_{\infty}\int_{s}^{t}M_{0,-\alpha}(f_{\epsilon,R}(z))M_{0,\beta}(f_{\epsilon,R}(z))\der z+C||\partial_{a}\varphi||_{\infty}\int_{s}^{t}\int_{(0,\infty)^{2}}f_{\epsilon,R}(\eta,z)\der \eta\der z\nonumber\\
    &\leq C|t-s|.
\end{align}

Using the fact that $f_{\epsilon,R}\in U_{\tilde{\epsilon},R},$ we can extend (\ref{equicontinuity}) to hold for functions $\varphi\in\compactfun$ and then for all $\varphi\in\vanishfun$. For details, see \cite[Proposition 3.18]{cristian2022coagulation}.

Combining the found equicontinuity in (\ref{equicontinuity}) with the uniform moment estimates, which are independent of $\tilde{\epsilon}, R$ and $\delta$, we conclude using Arzelà–Ascoli theorem that there exists a subsequence of $\{f_{\epsilon,R}\}$, which we do not relabel, and an $f_{\epsilon}\in\textup{C}([0,T];\mathscr{M}^{I}_{+}(\mathbb{R}_{>0}^{2})),$ such that $f_{\epsilon,R}(t)$ converge to $f_{\epsilon}(t)$ in the weak-$^{\ast}$ topology as $\tilde{\epsilon}\rightarrow 0$, $R\rightarrow\infty$ and $\delta\rightarrow 0$, for every $t\in[0,T].$  

Thus, we can use standard arguments found in the study of coagulation equations in order to pass to the limit as $\tilde{\epsilon}\rightarrow 0$ and $R\rightarrow\infty$ (and $\delta\rightarrow 0$ if $\mu>0$) in (\ref{truncated epislonic convergence}).
\end{proof}

\begin{prop}\label{prop ode}
Let $\mu>0$. We look at the system
\begin{align*}
\partial_{t}x_{A,V}^{\epsilon}(t)=\frac{1}{\epsilon}r(x_{A,V}^{\epsilon}(t),V)(c_{0}V^{\frac{2}{3}}-x_{A,V}^{\epsilon}(t)), & &x_{A,V}^{\epsilon}(0)=A.
 \end{align*}
 For any $(A,V)\in(0,\infty)^{2}$, fixed, and $t>0$, we have that $\lim_{\epsilon\rightarrow 0}x^{\epsilon}_{A,V}(t)=c_{0}V^{\frac{2}{3}}.$ The statement holds true for compact sets $\textup{K}$ of the form $(A,V,t)\in \textup{K}\subset (0,\infty)^{3}$. 
\end{prop}
\begin{proof}
Define $\tilde{x}_{A,V}(t):=x_{A,V}^{\epsilon}(\epsilon t)$. Then $\tilde{x}_{A,V}(0)=A$ and
\begin{align}\label{equation}
\partial_{t}\tilde{x}_{A,V}(t)=\partial_{t}x_{A,V}^{\epsilon}(\epsilon t)=\epsilon\partial_{\epsilon t}x_{A,V}^{\epsilon}(\epsilon t)=r(\tilde{x}_{A,V},V)(c_{0}V^{\frac{2}{3}}-\tilde{x}_{A,V})=:f(\tilde{x}_{A,V}(t)).
 \end{align}
 This implies that $\lim_{\epsilon\rightarrow 0}x^{\epsilon}_{A,V}(t)=\lim_{\epsilon\rightarrow 0}
 \tilde{x}_{A,V}(\frac{t}{\epsilon})=\lim_{t\rightarrow\infty}\tilde{x}_{A,V}(t).$
 In (\ref{equation}) notice that
 \begin{align*}
 \begin{cases}
 \tilde{x}_{A,V}(t)=c_{0}V^{\frac{2}{3}}\Rightarrow f(\tilde{x}_{A,V}(t))=0,\\
 \tilde{x}_{A,V}(t)<c_{0}V^{\frac{2}{3}}\Rightarrow f(\tilde{x}_{A,V}(t))>0,\\
 \tilde{x}_{A,V}(t)>c_{0}V^{\frac{2}{3}}\Rightarrow f(\tilde{x}_{A,V}(t))<0.
 \end{cases}
 \end{align*}  
 This implies that $\lim_{t\rightarrow\infty}\tilde{x}_{A,V}(t)=c_{0}V^{\frac{2}{3}}.$
\end{proof}
Notice that, at least for small times, the coagulation term in (\ref{weak_form_time_dependent}) gives a small contribution when we are away from the line $\{a=c_{0}v^{\frac{2}{3}}\}$. We thus look to control the contribution coming from the fusion term. For this, we first look at a simplified form for the adjoint problem of (\ref{weak_form_time_dependent}).

For simplicity, we denote by
\begin{align}\label{setisoperimetric}
\textup{S}:=\{(a,v)\in\mathbb{R}^{2}_{>0},a\geq c_{0}v^{\frac{2}{3}}\}.
\end{align}

\begin{prop}[Dual equation, case $\mu> 0$]\label{continuoussemigroup 2 positive mu}
Let $T>0$. Let $\epsilon\in(0,1)$, $R>1$ and $\mu>0$. Let $T>0$.
Let $\chi(\eta)$ be an arbitrary function in $\textup{C}_{\textup{b}}^{1}(\textup{S})$, where $\textup{S}$ is as in (\ref{setisoperimetric}), such that $\chi(\eta)=0, \textup{ when } v\not\in[\frac{1}{R},R]$. Then there exists a solution $\varphi_{\epsilon}\in \textup{W}_{T}$, with $\varphi_{\epsilon}(T,\cdot)=\chi(\cdot)$, which solves the following equation: 
\begin{align}\label{dual equation function}
\partial_{t}\varphi_{\epsilon}(t,\eta)+\frac{1}{\epsilon}\overline{r}_{\epsilon}(a,v)(c_{0}v^{\frac{2}{3}}-a)\partial_{a}\varphi_{\epsilon}(t,\eta)=0,
\end{align}
where
\begin{align*}
    \textup{W}_{T}:=\{\varphi\in\textup{C}^{1}([0,T],\textup{C}_{\textup{b}}^{1}(\textup{S}))|\varphi(t,\eta)=0, \textup{ when } v\not\in[\frac{1}{R},R], \textup{ for every } t\in[0,T]\}
\end{align*}
and
\begin{align*}
\overline{r}_{\epsilon}(a,v)=\frac{r(a,v)}{1+\epsilon^{2\mu+2}a^{\mu}}.
\end{align*}
 Assume in addition that there exists $\delta_{1}>0$ such that $\chi(\eta)=0$ when $a\leq c_{0}v^{\frac{2}{3}}+\delta_{1}$. Then the following statements hold:
\begin{enumerate}
    \item For every $M>1$ and $\overline{\sigma}>0$, there exists $\epsilon_{M,\overline{\sigma},\delta_{1}}\in(0,1)$ such that for every $\epsilon\leq \epsilon_{M,\overline{\sigma},\delta_{1}}$, we have that $\supp\varphi_{\epsilon}(t)\subseteq\{(a,v)\in(0,\infty)^{2}|a\geq M\}$, for every $t\in[0,T-\overline{\sigma}]$.
    \item There exists a constant $C(T),$ which can depend on time, but is independent on $\epsilon$, such that $\sup_{t\in[0,T]}\big(||\varphi_{\epsilon}(t)||_{\infty}+||\partial_{a}\varphi_{\epsilon}||_{\infty}\big)\leq C(T).$ Moreover, the constant $C(T)$ is independent of the initial datum $\chi$ if we assume $||\chi||_{\infty}+||\partial_{a}\chi||_{\infty}\leq 1.$
    \item If we assume $||\chi||_{\infty}+||\partial_{a}\chi||_{\infty}\leq 1$ and that $\epsilon$ is sufficiently small, then there exists a constant $C>0$, which is independent of $\epsilon>0$, but can depend on $R>1$, such that $||\partial_{a}\varphi_{\epsilon}(t)||_{\infty}\leq \textup{e}^{-\frac{C(T-t)}{\epsilon}}$.
\end{enumerate}
\end{prop}
\begin{proof}
For Statement $1$ we use Proposition \ref{prop ode}, the fact that we can find an explicit solution for equation (\ref{dual equation function}) and then we let $\epsilon\rightarrow 0$.

Statement $2$ follows directly from the fact that at time $T$ we have $\varphi_{\epsilon}(T,\cdot)=\chi(\cdot)\in\textup{C}_{\textup{b}}^{1}(\textup{S})$ and by integrating along the characteristics in equation (\ref{dual equation function}).

For Statement 3, we notice that, for $a\geq c_{0}v^{\frac{2}{3}}$, we have that
\begin{align}\label{derivative of fusion kernel}
    \partial_{a}\big[\frac{r(a,v)}{1+\epsilon^{2\mu+2}a^{\mu}}(a-c_{0}v^{\frac{2}{3}})\big]&=\partial_{a}\big[\frac{r(a,v)}{1+\epsilon^{2\mu+2}a^{\mu}}\big](a-c_{0}v^{\frac{2}{3}})+\frac{r(a,v)}{1+\epsilon^{2\mu+2}a^{\mu}}\nonumber\\
    &=\frac{\partial_{a}r(a,v)-\frac{\mu\epsilon^{2\mu+2}a^{\mu-1}r(a,v)}{1+\epsilon^{2\mu+2}a^{\mu}}}{1+\epsilon^{2\mu+2}a^{\mu}}(a-c_{0}v^{\frac{2}{3}})+\frac{r(a,v)}{1+\epsilon^{2\mu+2}a^{\mu}}\nonumber\\
    &=\frac{\partial_{a}r(a,v)-\mu a^{-1}r(a,v)\frac{\epsilon^{2\mu+2}a^{\mu}}{1+\epsilon^{2\mu+2}a^{\mu}}}{1+\epsilon^{2\mu+2}a^{\mu}}(a-c_{0}v^{\frac{2}{3}})+\frac{r(a,v)}{1+\epsilon^{2\mu+2}a^{\mu}}
    \nonumber\\
    &\geq\frac{\partial_{a}r(a,v)-\mu a^{-1}r(a,v)}{1+\epsilon^{2\mu+2}a^{\mu}}(a-c_{0}v^{\frac{2}{3}})+\frac{r(a,v)}{1+\epsilon^{2\mu+2}a^{\mu}}\nonumber\\
    &\geq\frac{r(a,v)}{1+\epsilon^{2\mu+2}a^{\mu}},
    \end{align}
where for the last inequality in (\ref{derivative of fusion kernel}), we used (\ref{ode_fusion}).

 We then analyse the cases $\epsilon^{2\mu+2}a^{\mu}\leq 1$ and $\epsilon^{2\mu+2}a^{\mu}\geq 1$ in order to deduce that $\frac{r(a,v)}{1+\epsilon^{2\mu+2}a^{\mu}}\geq C(R)\min\{\frac{a^{\mu}}{2},\frac{1}{2\epsilon^{2\mu+2}}\}$. By taking $\epsilon$ to be sufficiently small, we deduce from (\ref{derivative of fusion kernel}) that
\begin{align}\label{derivative of fusion kernel part 2}
    \partial_{a}\big[\frac{r(a,v)}{1+\epsilon^{2\mu+2}a^{\mu}}(a-c_{0}v^{\frac{2}{3}})\big]\geq C(R).
    \end{align}

We now look at the ODE
\begin{align*}
    \partial_{t}x(\eta,t)=\frac{1}{\epsilon}\overline{r}_{\epsilon}(x,v)(c_{0}v^{\frac{2}{3}}-x) \textup{ with } x(\eta,T)=a,
\end{align*}
which, by taking $s=\frac{T-t}{\epsilon}$, reduces to solving 
\begin{align*}
    \partial_{s}x(\eta,s)=\overline{r}_{\epsilon}(x,v)(x-c_{0}v^{\frac{2}{3}}) \textup{ with } x(\eta,0)=a.
\end{align*}
Using (\ref{derivative of fusion kernel part 2}), we obtain that
\begin{align*}
\partial_{a}x(\eta,s)\geq \textup{e}^{C(R)s}
\end{align*}
and Statement 3 follows using in addition the fact that $||\partial_{a}\chi||_{\infty}\leq 1.$
\end{proof}
\begin{rmk}
In the case $\mu\leq 0$, the condition (\ref{ode_fusion}) is more general due to the fact that we will not modify the fusion term $r$ in Proposition \ref{continuoussemigroup 2 positive mu}.
\end{rmk}
\begin{prop}\label{extension compact support}
Let $\epsilon>0$, fixed and $\mu>0$. Then equation (\ref{epsilonic_convergence}) holds for every $\varphi\in\textup{C}^{1}([0,T];\textup{C}^{1}(\mathbb{R}_{>0}^{2}))$ with $\sup_{s\in[0,T];\eta\in\mathbb{R}_{>0}^{2}}|\partial_{t}\varphi(\eta,t)+\varphi(\eta,t)+\partial_{a}\varphi(\eta,t)|\leq C$ if
\begin{align}\label{moment estimates to lose compact support}
    \sup_{t\in[0,T]} \int_{(0,\infty)^{2}}(a^{\mu+2}+v^{\max\{\sigma(\mu+2),2\}}+v^{\min\{\sigma(\mu+2),-1\}})f_{\epsilon}(a,v,t)\leq C(T).
\end{align}
\end{prop}
\begin{proof}
Assume for simplicity that $\varphi\in \textup{C}_{\textup{b}}^{1}(\mathbb{R}_{>0}^{2})$. We construct a sequence of functions $\{\zeta_{n}\}_{n\in\mathbb{N}}\subset \textup{C}_{\textup{c}}^{1}(\mathbb{R}_{>0}^{2})$ such that $\zeta_{n}(\eta)=1$ when $\eta\in[\frac{1}{n},n]^{2}$ and $\zeta_{n}(\eta)=0$ when  $\eta\not\in[\frac{1}{2n},2n]^{2}$. The idea is to use Lebegue's dominated convergence theorem in (\ref{weak_form_time_dependent}) for the functions $\varphi_{n}=\zeta_{n}\varphi$. We thus show below only the needed estimates for the proof. The term with the coagulation kernel in (\ref{weak_form_time_dependent}) can be bounded directly by
\begin{align*}
    \big|\langle \mathbb{K}[f_{\epsilon}],\varphi_{n}\rangle \big|\leq 3C\sup_{s\in[0,T]} M_{0,-\alpha}(f(s))\sup_{s\in[0,T]} M_{0,\beta}(f(s))\leq C.
\end{align*}
 In order to control the fusion term in (\ref{weak_form_time_dependent}), notice that we can construct $\zeta_{n}$ such that $a\zeta_{n}(\eta)\leq C,$ for some constant independent of $n\in\mathbb{N}.$ Moreover, we know that the fusion kernel satisfies (\ref{fusion_form}) and that $a\geq c_{0}v^{\frac{2}{3}}$. Thus
 \begin{align*}
     |r(a,v)(c_{0}v^{\frac{2}{3}}-a)|\lesssim a^{\mu+1}v^{\sigma}+a^{\mu}v^{\sigma+\frac{2}{3}}\leq 2 a^{\mu+1}v^{\sigma}.
 \end{align*}
 Using the above inequality, we can bound from above the fusion term 
 \begin{align*}
     \frac{1}{\epsilon}\big|\int_{(0,\infty)^{2}}r(\eta)(c_{0}v^{\frac{2}{3}}-a)f_{\epsilon}(\eta,s)\partial_{a}\varphi_{n}(\eta)\der \eta\big|
 \end{align*} 
 by 
 \begin{align*}
\frac{1}{\epsilon}\big|\int_{(0,\infty)^{2}}a^{\mu+1}v^{\sigma}f_{\epsilon}(\eta,s)\partial_{a}\varphi(\eta)\der \eta\big|+\frac{1}{\epsilon}\big|\int_{(0,\infty)^{2}}a^{\mu+1}v^{\sigma}f_{\epsilon}(\eta,s)\varphi(\eta)\partial_{a}\zeta_{n}(\eta)\der \eta\big|
\end{align*}
up to a multiplicity constant. We then use Young's inequality to deduce that
\begin{align*}
\big|\int_{(0,\infty)^{2}}a^{\mu+1}v^{\sigma}f_{\epsilon}(\eta,s)\partial_{a}\varphi(\eta)\der \eta\big|\lesssim \sup_{s\in[0,T]}M_{\mu+2,0}(f_{\epsilon}(s))+\sup_{s\in[0,T]}M_{0,\sigma(\mu+2)}(f_{\epsilon}(s))
\end{align*}
and
\begin{align*}
\big|\int_{(0,\infty)^{2}}a^{\mu+1}v^{\sigma}f_{\epsilon}(\eta,s)\varphi(\eta)\partial_{a}\zeta_{n}(\eta)\der \eta\big|&\lesssim M_{\mu,\sigma}(f_{\epsilon}(s))\\
&\lesssim \sup_{s\in[0,T]}M_{\mu+1,0}(f_{\epsilon}(s))+\sup_{s\in[0,T]}M_{0,\sigma(\mu+2)}(f_{\epsilon}(s))+1.
\end{align*}
Thus, the moment estimates in (\ref{moment estimates to lose compact support}) suffice to conclude our proof for $\varphi\in \textup{C}_{\textup{b}}^{1}(\mathbb{R}_{>0}^{2})$. To prove that equation (\ref{epsilonic_convergence}) holds for every $\varphi\in\textup{C}^{1}([0,T];\textup{C}^{1}(\mathbb{R}_{>0}^{2}))$ with $\sup_{s\in[0,T];\eta\in\mathbb{R}_{>0}^{2}}|\partial_{t}\varphi(\eta,t)+\varphi(\eta,t)+\partial_{a}\varphi(\eta,t)|\leq C$, we argue in a similar manner as before using the bound on the time derivative too.
\end{proof}
\begin{proof}[Proof of Proposition \ref{solutions stay close to the isoperimetric line}]
Let $\overline{\sigma}>0$ and $\delta_{1},\delta_{2}\in(0,1)$. Assume $\Phi\in\textup{C}^{1}_{\textup{b}}(\mathbb{R}_{>0}^{2})$, with $||\Phi||_{\infty}\leq 1$. Assume there exists $R>1$ such that $\Phi(\eta)=0$ if $v\not\in[\frac{1}{R},R].$ The extension of the result from functions compactly supported in the $v$ variable to functions that do not necessarily have compact support is straightforward using moment estimates, and thus we omit the details.  Suppose in addition that $\Phi$ is such that $\Phi(\eta)=0$ when $a< c_{0}v^{\frac{2}{3}}+\delta_{1}$. Let, in addition, $t_{1},t_{2}$, with $t_{2}\geq \overline{\sigma}$, and $t_{1}$ such that $t_{2}-t_{1}=\tau>0$. Notice that $\tau=\tau(\delta_{2})$ depends on $\tau_{2}$, but we will not write this dependence explicitly in order to simplify our notation. Let $M>0,$ sufficiently large (and also depending on $\delta_{2}$), to be fixed later, and let $\varphi_{\epsilon}$ be the solution found in Proposition \ref{continuoussemigroup 2 positive mu}, if $\mu>0$, associated to the measure $f_{\epsilon}$ such that $\varphi_{\epsilon}(\eta,t_{2})=\Phi(\eta).$ We want to prove (\ref{prop small outside iso line}).

Notice first that, from Proposition \ref{continuoussemigroup 2 positive mu}, Statement 2, there exists a constant independent of $\epsilon,$ such that $\sup_{s\in[t_{1},t_{2}],\eta\in\textup{S}}|\varphi(\eta,s)|\leq C$. Then, for $t_{1},t_{2}\in(0,T]$ such that $t_{2}-t_{1}=\tau,$ we have
\begin{align}\label{kernel term is small}
&\big|\int_{t_{1}}^{t_{2}}\int_{(0,\infty)^{2}}\int_{(0,\infty)^{2}} K(\eta,\eta')f(\eta,s)f(\eta',s)[\varphi_{\epsilon}(\eta+\eta',s)-\varphi_{\epsilon}(\eta,s)-\varphi_{\epsilon}(\eta',s)]\der \eta'\der\eta\der s\big|\nonumber\\
&\leq C\tau\sup_{s\in[0,T]}[M_{0,-\alpha}(f(s)) M_{0,\beta}(f(s))]\leq C\tau,
\end{align}
where we made use of the fact that $[M_{0,1}+M_{0,-1}](f_{\epsilon})$ is uniformly bounded from above, independently of $\epsilon\in(0,1)$.

In order to estimate the term with the fusion kernel, we notice that
\begin{align}\label{fusion term small}
\frac{1}{\epsilon}\big|\int_{t_{1}}^{t_{2}}&\int_{(0,\infty)^{2}}[r(a,v)-\overline{r}_{\epsilon}(a,v)](c_{0}v^{\frac{2}{3}}-a)f(\eta,s)\partial_{a}\varphi_{\epsilon}(\eta,s)\der\eta\der s\big|\nonumber\\
&\leq \frac{1}{\epsilon}\big|\int_{t_{1}}^{t_{2}}\int_{\{a\leq \frac{1}{\epsilon}\}}[r(a,v)-\overline{r}_{\epsilon}(a,v)](c_{0}v^{\frac{2}{3}}-a)f(\eta,s)\partial_{a}\varphi_{\epsilon}(\eta,s)\der\eta\der s\big|\nonumber\\
&+\frac{2}{\epsilon}\int_{t_{1}}^{t_{2}}\int_{\{a>\frac{1}{\epsilon}\}}r(a,v)(c_{0}v^{\frac{2}{3}}+a)f(\eta,s)|\partial_{a}\varphi_{\epsilon}(\eta,s)|\der\eta\der s.
\end{align}
For the first term in (\ref{fusion term small}), we make use of the exponential decay of  $\varphi_{\epsilon}$ proven in Statement 3  of Proposition \ref{continuoussemigroup 2 positive mu} in order to obtain
\begin{align}\label{small positive distance iso line area of order one}
\frac{1}{\epsilon}\big|\int_{t_{1}}^{t_{2}}&\int_{\{a\leq \frac{1}{\epsilon}\}}[r(a,v)-\overline{r}_{\epsilon}(a,v)](c_{0}v^{\frac{2}{3}}-a)f(\eta,s)\partial_{a}\varphi_{\epsilon}(\eta,s)\der\eta\der s\big|\nonumber\\
&\lesssim\int_{t_{1}}^{t_{2}}\int_{\{a\leq\frac{1}{\epsilon}\}}|r(a,v)-\overline{r}_{\epsilon}(a,v)|(c_{0}v^{\frac{2}{3}}+a)f(\eta,s)\frac{\textup{e}^{-\frac{C(t_{2}-s)}{\epsilon}}}{\epsilon}\der\eta\der s\nonumber\\
&\lesssim 2R_{1} \int_{t_{1}}^{t_{2}}\int_{\{a\leq \frac{1}{\epsilon}\}}\epsilon^{2\mu+2}a^{2\mu+1}v^{\sigma}f(\eta,s)\frac{\textup{e}^{-\frac{C(t_{2}-s)}{\epsilon}}}{\epsilon}\der\eta\der s\nonumber\\
&\leq 2R_{1}\epsilon\int_{t_{1}}^{t_{2}}\sup_{s\in[0,T]}M_{0,\sigma}(f_{\epsilon}(t))\frac{\textup{e}^{-\frac{C(t_{2}-s)}{\epsilon}}}{\epsilon}\der s\leq 2R_{1}\epsilon C(T)\int_{t_{1}}^{t_{2}}\frac{\textup{e}^{-\frac{C(t_{2}-s)}{\epsilon}}}{\epsilon}\der s\nonumber\\
&\leq C\epsilon.
\end{align}

We can then use Statement 2 of Proposition \ref{continuoussemigroup 2 positive mu}, the upper bound (\ref{fusion_form}) and (\ref{upper estimates for region positive distance from iso line}) in order to control the region containing large values of $a$, namely
\begin{align}\label{small positive distance iso line large values area}
  \frac{2}{\epsilon}\int_{t_{1}}^{t_{2}} &\int_{\{a>\frac{1}{\epsilon}\}}r(a,v)(c_{0}v^{\frac{2}{3}}+a)f(\eta,s)|\partial_{a}\varphi_{\epsilon}(\eta,s)|\der\eta\der s\nonumber\\
       &\leq \frac{4R_{1}}{\epsilon}\int_{t_{1}}^{t_{2}}\int_{\{a>\frac{1}{\epsilon}\}}a^{\mu+1}v^{\sigma}f(\eta,s)\der\eta\der s\leq \frac{4R_{1}\epsilon}{\epsilon}\int_{t_{1}}^{t_{2}}\int_{\{a>\frac{1}{\epsilon}\}}a^{\mu+2}v^{\sigma}f(\eta,s)\der\eta\der s\nonumber\\
             &\leq 4R_{1}\tau\sup_{t\in[0,T]}[M_{\mu+3,0}(f_{\epsilon}(t))+M_{0,\sigma(\mu+3)}(f_{\epsilon}(t))].
\end{align}
Combining the estimates (\ref{fusion term small}), (\ref{small positive distance iso line area of order one}) and (\ref{small positive distance iso line large values area}), we deduce that
\begin{align}\label{fusion term small final}
\frac{1}{\epsilon}\big|\int_{t_{1}}^{t_{2}}&\int_{(0,\infty)^{2}}[r(a,v)-\overline{r}_{\epsilon}(a,v)](c_{0}v^{\frac{2}{3}}-a)f(\eta,s)\partial_{a}\varphi_{\epsilon}(\eta,s)\der\eta\der s\big|\leq C\epsilon+C\tau.
\end{align}

Proposition \ref{extension compact support} gives us that we can test (\ref{weak_form_time_dependent}) with continuous functions that are not necessarily compactly supported in the $a$ variable, as long as we work with functions in  $\textup{C}^{1}_{\textup{b}}(\mathbb{R}_{>0}^{2})$. We then make use of the fact that $\varphi_{\epsilon}$ satisfies (\ref{dual equation function}) and then use the estimates in (\ref{kernel term is small}) and (\ref{fusion term small final}) in order to deduce that, for every $\epsilon\leq \epsilon_{\delta_{1},\delta_{2}}$, with $\epsilon_{\delta_{1},\delta_{2}}$ (notice here that $M$ and $\tau$ depend on $\delta_{2}$) as in Proposition \ref{continuoussemigroup 2 positive mu}, Statement 1, we have
\begin{align}\label{sketch proof mass concentration}
  \big| \int_{(0,\infty)^{2}}f_{\epsilon}(\eta,t_{2})\Phi(a,v)\der \eta\big|&= \big|\int_{(0,\infty)^{2}}f_{\epsilon}(\eta,t_{2})\varphi_{\epsilon}(\eta,t_{2})\der \eta\big|\leq\big|\int_{(0,\infty)^{2}}f_{\epsilon}(\eta,t_{1})\varphi_{\epsilon}(\eta,t_{1})\der \eta\big|+C\tau\nonumber \\
    &=\big|\int_{\{a\geq M\}}f_{\epsilon}(\eta,t_{1})\varphi_{\epsilon}(\eta,t_{1})\der \eta\big|+C\tau\nonumber\nonumber\\
&\leq CM^{-1}\int_{\{a\geq M\}}af_{\epsilon}(\eta,t_{1})\der \eta+C\tau,
\end{align}
where we used that  $\supp\varphi_{\epsilon}\subseteq\{a\geq M\}$ from Statement 1 of Proposition \ref{continuoussemigroup 2 positive mu} and the fact that $\sup_{t\in[0,T]}M_{1,0}(f_{\epsilon}(t))\leq C(T)$, with $C(T)$ being independent of $\epsilon$. We choose $M$ such that $C(T)CM^{-1}+C\tau\leq\delta_{2}$.  

Thus, we obtain that
\begin{align}\label{far away iso line c1 functions}
\big|\int_{(0,\infty)^{2}}f_{\epsilon}(\eta,t_{2})\Phi(\eta)\der \eta\big|\leq \delta_{2},
\end{align}
for every $t_{2}\geq \overline{\sigma}$ since we can choose $\tau$ to be sufficiently small.

\end{proof}
We are now able to prove the equicontinuity in time of the sequence $\{f_{\epsilon}\}_{\epsilon\in(0,1)}$. We proved in Proposition \ref{solutions stay close to the isoperimetric line} that if we are at a positive distance from the line $a=c_{0}v^{\frac{2}{3}}$, the measure $f_{\epsilon}$ takes values close to zero. Near the line $a=c_{0}v^{\frac{2}{3}},$ we use the fact that a function of the form $\varphi(a,v)$ can be approximated in terms of a function depending only on $v$, making negligible the contribution coming from the fusion term.
\begin{proof}[Proof of Proposition \ref{existence of limit for f epsilon}]
Fix $\overline{\sigma}>0$. Let $\delta_{1}, \delta_{2}>0$ as in Proposition \ref{solutions stay close to the isoperimetric line}. We define a continuous function $\chi_{\delta_{1}}:\mathbb{R}_{>0}^{2}\rightarrow[0,1]$ to be equal to one in the region where $\{c_{0}v^{\frac{2}{3}}\leq a \leq c_{0}v^{\frac{2}{3}}+\delta_{1}\}$ and zero when $a>c_{0}v^{\frac{2}{3}}+2\delta_{1}$. Fix $\varphi\in\textup{C}^{1}_{\textup{c}}(\mathbb{R}_{>0}^{2})$. It suffices to prove the statement for fixed $\varphi\in\textup{C}^{1}_{\textup{c}}(\mathbb{R}_{>0}^{2})$. The passage from functions $\varphi\in\textup{C}^{1}_{\textup{c}}(\mathbb{R}_{>0}^{2})$ to functions in $\textup{C}_{0}(\mathbb{R}_{>0}^{2})$ is then straightforward using moment estimates and the fact that we can approximate a function in $\textup{C}_{\textup{c}}(\mathbb{R}_{>0}^{2})$ with a $\textup{C}^{1}_{\textup{c}}(\mathbb{R}_{>0}^{2})$ function on a compact set. From Proposition \ref{solutions stay close to the isoperimetric line}, we deduce that there exists $\epsilon_{\delta_{1},\delta_{2}}\in(0,1)$ such that for all $\epsilon\leq \epsilon_{\delta_{1},\delta_{2}},$ the following holds
\begin{align}\label{first part equicontinuity}
\big|&\int_{(0,\infty)^{2}}\varphi(\eta)f_{\epsilon}(\der \eta,t)-\int_{(0,\infty)^{2}}\varphi(\eta)f_{\epsilon}(\der \eta,s)\big|\nonumber\\
&\leq\big|\int_{(0,\infty)^{2}}\varphi(\eta)\chi_{\delta_{1}}(\eta)[f_{\epsilon}(\der \eta,t)-f_{\epsilon}(\der \eta,s)]\big|+\big|\int_{(0,\infty)^{2}}\varphi(\eta)(1-\chi_{\delta_{1}}(\eta))[f_{\epsilon}(\der \eta,t)-f_{\epsilon}(\der \eta,s)]\big|\nonumber\\
&\leq\big|\int_{(0,\infty)^{2}}\varphi(\eta)\chi_{\delta_{1}}(\eta)[f_{\epsilon}(\der \eta,t)-f_{\epsilon}(\der \eta,s)]\big|+\delta_{2}.
\end{align}
From the support of $\chi_{\delta_{1}}$, the fact that the measure $f_{\epsilon}$ is supported in the region $\{a\geq c_{0}v^{\frac{2}{3}}\}$ and the continuity of $\varphi$, we can find a function $\overline{\varphi}:\mathbb{R}_{>0}\rightarrow\mathbb{R}$ depending only on $v$, $\overline{\varphi}\in\textup{C}_{\textup{b}}^{1}(\mathbb{R}_{>0})$, such that $|\varphi(a,v)-\overline{\varphi}(v)|\leq \delta_{3}$, for $\delta_{3}$ sufficiently small, and for any $(a,v)\in\textup{K}\subset\mathbb{R}_{>0}^{2}$, where $\textup{K}$ is a compact set. Let $K=[\frac{1}{M},{M}]^{2}$ for some fixed $M>1$. We have that
\begin{align}\label{second part equicontinuity}
&\big|\int_{(0,\infty)^{2}}\varphi(\eta)\chi_{\delta_{1}}(\eta)[f_{\epsilon}(\der \eta,t)-f_{\epsilon}(\der \eta,s)]\bigg|\nonumber\\
&\leq \int_{(0,\infty)^{2}}|\varphi(\eta)-\overline{\varphi}(v)|\chi_{\delta_{1}}(\eta)|f_{\epsilon}(\der \eta,t)-f_{\epsilon}(\der \eta,s)|+\big|\int_{(0,\infty)^{2}}\overline{\varphi}(v)\chi_{\delta_{1}}(\eta)[f_{\epsilon}(\der \eta,t)-f_{\epsilon}(\der \eta,s)]\big|\nonumber\\
&\leq \int_{K}|\varphi(\eta)-\overline{\varphi}(v)|\chi_{\delta_{1}}(\eta)|f_{\epsilon}(\der \eta,t)-f_{\epsilon}(\der \eta,s)|\nonumber\\
&+\int_{\mathbb{R}_{>0}^{2}\setminus K}|\varphi(\eta)-\overline{\varphi}(v)|\chi_{\delta_{1}}(\eta)|f_{\epsilon}(\der \eta,t)-f_{\epsilon}(\der \eta,s)|+\big|\int_{(0,\infty)^{2}}\overline{\varphi}(v)\chi_{\delta_{1}}(\eta)[f_{\epsilon}(\der \eta,t)-f_{\epsilon}(\der \eta,s)]\big|.\nonumber
\end{align}
In order to control the contribution of the region $\mathbb{R}_{>0}^{2}\setminus K$, we use moment estimates. For example, for the region $\{a>M\}$, we obtain that 
\begin{align*}
    \int_{\{a>M\}}|\varphi(\eta)-\overline{\varphi}(v)|\chi_{\delta_{1}}(\eta)|f_{\epsilon}(\der \eta,t)-f_{\epsilon}(\der \eta,s)|\lesssim 2M^{-1} \sup_{s\in[0,T]}M_{1,0}(f_{\epsilon}(s)).
\end{align*}
We can use the same argument for the region $\{v>M\}\cup\{v<\frac{1}{M}\}\cup\{a<\frac{1}{M}\}$. Thus
\begin{align}
&\big|\int_{(0,\infty)^{2}}\varphi(\eta)\chi_{\delta_{1}}(\eta)[f_{\epsilon}(\der \eta,t)-f_{\epsilon}(\der \eta,s)]\bigg|\nonumber\\
&\lesssim 2\delta_{3} \sup_{s\in[0,T]}M_{0,0}(f_{\epsilon}(s))+2M^{-1} \sup_{s\in[0,T]}M_{1,0}(f_{\epsilon}(s))+\big|\int_{(0,\infty)^{2}}\overline{\varphi}(v)\chi_{\delta_{1}}(\eta)[f_{\epsilon}(\der \eta,t)-f_{\epsilon}(\der \eta,s)]\big|\nonumber\\
&\leq C(T)\delta_{3}+\big|\int_{(0,\infty)^{2}}\overline{\varphi}(v)\chi_{\delta_{1}}(\eta)[f_{\epsilon}(\der \eta,t)-f_{\epsilon}(\der \eta,s)]\big|,
\end{align}
since the term $M^{-1} \sup_{s\in[0,T]}M_{1,0}(f_{\epsilon}(s))$ can be made sufficiently small by taking $M$ sufficiently large. We can write the remaining term as $\overline{\varphi}(v)\chi_{\delta_{1}}(\eta)=\overline{\varphi}(v)-\overline{\varphi}(v)(1-\chi_{\delta_{1}}(\eta))$. For $\overline{\varphi}(v)(1-\chi_{\delta_{1}}(\eta))$ we can make use again of Proposition  \ref{solutions stay close to the isoperimetric line} in order to prove that $\big|\int_{(0,\infty)^{2}}\overline{\varphi}(v)(1-\chi_{\delta_{1}}(\eta))[f_{\epsilon}(\der \eta,t)-f_{\epsilon}(\der \eta,s)]\big|$ is small. On the other hand, we deduce from Proposition \ref{extension compact support} that we can test equation (\ref{epsilonic_convergence}) with $\overline{\varphi}(v)$ and obtain, as the fusion term disappears when testing with functions depending only on $v$ that
\begin{align}\label{third part equicontinuity}
\big|\int_{(0,\infty)^{2}}\overline{\varphi}(v)[f_{\epsilon}(\der \eta,t)-f_{\epsilon}(\der \eta,s)]\big|\leq 3||\overline{\varphi}||_{\infty}\int_{s}^{t}M_{0,-\alpha}(f_{\epsilon}(z))M_{0,\beta}(f_{\epsilon}(z))\der z\leq C|t-s|.
\end{align}
From (\ref{first part equicontinuity}), (\ref{second part equicontinuity}) and (\ref{third part equicontinuity}) the equicontinuity in time of the sequence follows for every $s,t\in[\overline{\sigma},T].$
\end{proof}
We are now in a position to prove Theorem \ref{existencelimitandcorrectsupport}.
\begin{proof}[Proof of Theorem \ref{existencelimitandcorrectsupport}]
Let $\overline{\sigma}>0$. Using Proposition \ref{existence of solutions for the equation} together with Proposition \ref{existence of limit for f epsilon}, we deduce from Arzelà–Ascoli theorem that there exists a subsequence of $\{f_{\epsilon}\}$, which we do not relabel, and an $\overline{f}\in\textup{C}([\overline{\sigma},T];\mathscr{M}^{I}_{+}(\mathbb{R}_{>0}^{2})),$ such that $f_{\epsilon}(t)$ converges to $f(t)$ in the weak-$^{\ast}$ as $\epsilon\rightarrow 0$, for every $t\in[\overline{\sigma},T].$
\end{proof}
\begin{proof}[Proof of Lemma \ref{lemma dirac measure}]
     We pass to the limit in Proposition \ref{solutions stay close to the isoperimetric line}.
\end{proof}
\subsection{Reduction to the one-dimensional coagulation model in the case of fast fusion}\label{reduction to the one-dimensional case section}
We now consider the behavior of the solutions of (\ref{1 over epsilon strongfusioneq}) as $\Lambda\rightarrow 0$. We first show that we can extend $F$ in Lemma \ref{lemma dirac measure} to be continuous at time $t=0$ as mentioned in Theorem \ref{pass back to standard coagulation}.
\begin{prop}\label{behavior near initial data}
Let $T>0$. Let $F_{\epsilon}\in\textup{C}([0,T];\mathscr{M}_{+}(\mathbb{R}_{>0}))$ be as in Definition \ref{limit function}. Then the sequence $\{F_{\epsilon}\}_{\epsilon\in(0,1)}$ is equicontinuous in time, and thus we can deduce that there exists a limit $F\in\textup{C}([0,T];\mathscr{M}_{+}(\mathbb{R}_{>0}))$ as $\epsilon\rightarrow 0$ for a subsequence of $\{F_{\epsilon}\}_{\epsilon\in(0,1)}$, which we do not relabel.
\end{prop}
\begin{proof}
Let $s,t\in[0,T]$, with $t>s$. Fix $\varphi\in\textup{C}_{0}(\mathbb{R}_{>0})$. We then have that
\begin{align*}
  &\big|\int_{(0,\infty)}\varphi(v)F_{\epsilon}(\der v,t)-\int_{(0,\infty)}\varphi(v)F_{\epsilon}(\der v,s)\big|\\
  &=\big|\int_{(0,\infty)^{2}}f_{\epsilon}(a,v,t)\chi_{\epsilon}(a)\varphi(v)\der v \der a-\int_{(0,\infty)^{2}}f_{\epsilon}(a,v,s)\chi_{\epsilon}(a)\varphi(v)\der v \der a \big|\\
&\leq \int_{s}^{t}\langle\mathbb{K}[f_{\epsilon}(z)],\chi_{\epsilon}\varphi\rangle\der z+\frac{1}{\epsilon}\int_{s}^{t}\int_{(0,\infty)^{2}}r(a,v)|a+c_{0}v^{\frac{2}{3}}||\partial_{a}\chi_{\epsilon}(a)\varphi(v)|f_{\epsilon}(\eta,z)\der \eta\der z\\
&\lesssim |t-s|[\sup_{z\in[0,T]}M_{0,-\alpha}(f_{\epsilon}(z))\sup_{z\in[0,T]}M_{0,\beta}(f_{\epsilon}(z))+\epsilon \big(\sup_{z\in[0,T]}M_{\mu+1,\sigma}(f_{\epsilon}(z))+\sup_{z\in[0,T]}M_{\mu,\sigma+\frac{2}{3}}(f_{\epsilon}(z))\big)]\\
&\leq  |t-s|[C(T)+C\epsilon\big(\sup_{z\in[0,T]}M_{\mu+2,0}(f_{\epsilon}(z))+\sup_{z\in[0,T]}M_{0,\sigma(\mu+2)}(f_{\epsilon}(z))\big)]\lesssim |t-s|.
\end{align*}
\end{proof}
\begin{lem}
 Let $\overline{\sigma}>0$. Let $F$ as in Theorem \ref{pass back to standard coagulation}. We then have that 
 $F$ satisfies the standard coagulation equation, namely, for every $t\in[\overline{\sigma},T]$ and every $\varphi\in\textup{C}_{\textup{c}}(\mathbb{R}_{>0})$, the following holds
\begin{align}\label{equation with sigma}
&\int_{(0,\infty)}F(v,t)\varphi(v)\der v-\int_{(0,\infty)}F(v,\overline{\sigma})\varphi(v)\der v\nonumber\\
&=\int_{\overline{\sigma}}^{t}\int_{(0,\infty)}K(c_{0}v^{\frac{2}{3}},v,c_{0}v'^{\frac{2}{3}},v')F(v,s)F(v',s)[\varphi(v+v')-\varphi(v)-\varphi(v')]\der v'\der v\der s.
\end{align}
\end{lem}
\begin{proof}
Let $\delta_{1},\delta_{2}\in(0,1)$. By Proposition \ref{solutions stay close to the isoperimetric line}, we have that there exists $\epsilon_{\delta_{1}, \delta_{2}}\in(0,1)$, such that
\begin{align}\label{main theo isoperimetric line}
\int_{\{a'>c_{0}v'^{\frac{2}{3}}+\delta_{1}\}}\Phi(\eta')f_{\epsilon}(\eta',t) \der \eta'\leq C\delta_{2},
\end{align}
for every $t\geq \overline{\sigma}>0$, every $\epsilon\leq\epsilon_{\delta_{1}, \delta_{2}}$ and for every $\Phi \in\textup{C}_{0}(\mathbb{R}_{>0}^{2})$ such that $||\Phi||_{\infty}\leq 1.$

Due to Proposition \ref{extension compact support}, we can test (\ref{epsilonic_convergence}) with $\varphi(a,v)=\tilde{\varphi}(v)$, where $\tilde{\varphi}\in\textup{C}^{1}_{0}(\mathbb{R}_{>0})$. Let $\delta_{1}$ as in (\ref{main theo isoperimetric line}). Denoting by $\chi_{\tilde{\varphi}}(v,v'):=\tilde{\varphi}(v+v')-\tilde{\varphi}(v)-\tilde{\varphi}(v')$, we have that
\begin{align}\label{main one dim}
&\int_{(0,\infty)^{2}}f_{\epsilon}(\eta,t)\tilde{\varphi}(v)\der \eta-\int_{(0,\infty)^{2}}f_{\epsilon}(\eta,\overline{\sigma})\tilde{\varphi}(v)\der \eta\nonumber\\
&=\frac{1}{2}\int_{\overline{\sigma}}^{t}\int_{(0,\infty)^{2}}\int_{(0,\infty)^{2}}K(\eta,\eta')[\tilde{\varphi}(v+v')-\tilde{\varphi}(v)-\tilde{\varphi}(v')]f_{\epsilon}(\eta,s)f_{\epsilon}(\eta',s)\der \eta \der \eta'\der s\nonumber\\
&=\frac{1}{2}\int_{\overline{\sigma}}^{t}\int_{\{a\in \textup{T}(v)\}}\int_{\{a'\in \textup{T}(v')\}}K(\eta,\eta')\chi_{\tilde{\varphi}}(v,v')f_{\epsilon}(\eta,s)f_{\epsilon}(\eta',s)\der \eta \der \eta'\der s\nonumber\\
&+\int_{\overline{\sigma}}^{t}\int_{\{a\in \textup{T}(v)\}}\int_{\{a'>c_{0}v'^{\frac{2}{3}}+\delta_{1}\}}K(\eta,\eta')\chi_{\tilde{\varphi}}(v,v')f_{\epsilon}(\eta,s)f_{\epsilon}(\eta',s)\der \eta \der \eta'\der s\nonumber\\
&+\frac{1}{2}\int_{\overline{\sigma}}^{t}\int_{\{a>c_{0}v^{\frac{2}{3}}+\delta_{1}\}}\int_{\{a'>c_{0}v'^{\frac{2}{3}}+\delta_{1}\}}K(\eta,\eta')\chi_{\tilde{\varphi}}(v,v')f_{\epsilon}(\eta,s)f_{\epsilon}(\eta',s)\der \eta \der \eta'\der s,
\end{align}
where we denoted $\textup{T}(v):=[c_{0}v^{\frac{2}{3}};c_{0}v^{\frac{2}{3}}+\delta_{1}]$. 

\textbf{Step 1.}
We first prove that, for every $t\geq \overline{\sigma}$, we have that
\begin{align*}
\frac{1}{2}\int_{\{a\in\textup{T}(v)\}}\int_{\{a'\in\textup{T}(v')\}}K(a,v,a',v')[\tilde{\varphi}(v+v')-\tilde{\varphi}(v)-\tilde{\varphi}(v')]f_{\epsilon}(\eta,t)f_{\epsilon}(\eta',t)\der \eta \der \eta'
\end{align*}
tends to
\begin{align*}
\frac{1}{2}\int_{(0,\infty)^{2}}K(c_{0}v^{\frac{2}{3}},v,c_{0}v'^{\frac{2}{3}},v')[\tilde{\varphi}(v+v')-\tilde{\varphi}(v)-\tilde{\varphi}(v')]F(v,t)F(v',t)\der v' \der v
\end{align*}
as $\epsilon\rightarrow 0.$

Step 1. a) We begin by proving that we are able to approximate the coagulation kernel near the isoperimetric line with a function only depending on the  $v$ variable, due to its continuity. In other words,
\begin{align*}
&3||\varphi||_{\infty}\int_{\{a\in\textup{T}(v)\}}\int_{\{a'\in\textup{T}(v')\}}|K(a,v,a',v')-K(c_{0}v^{\frac{2}{3}},v,c_{0}v'^{\frac{2}{3}},v')|f_{\epsilon}(\eta,t)f_{\epsilon}(\eta',t)\der \eta \der \eta'\\
&\leq C\int_{\{a\in\textup{T}(v);v\in[\frac{1}{M};M]\}}\int_{\{a'\in\textup{T}(v');v'\in[\frac{1}{M};M]\}}|K(a,v,a',v')-K(c_{0}v^{\frac{2}{3}},v,c_{0}v'^{\frac{2}{3}},v')|f_{\epsilon}(\eta,t)f_{\epsilon}(\eta',t)\der \eta \der \eta'\\
&+C\int_{\{v>M\}}\int|K(a,v,a',v')-K(c_{0}v^{\frac{2}{3}},v,c_{0}v'^{\frac{2}{3}},v')|f_{\epsilon}(\eta,t)f_{\epsilon}(\eta',t)\der \eta \der \eta'\\
&+C\int_{\{v<\frac{1}{M}\}}\int|K(a,v,a',v')-K(c_{0}v^{\frac{2}{3}},v,c_{0}v'^{\frac{2}{3}},v')|f_{\epsilon}(\eta,t)f_{\epsilon}(\eta',t)\der \eta \der \eta' =: \textup{I}_{1}+\textup{I}_{2}+\textup{I}_{3}.
\end{align*}
We can prove that $\textup{I}_{2}+\textup{I}_{3}$ gives a small contribution due to moment estimates. For $\textup{I}_{1}$, we use the fact that $|a-c_{0}v^{\frac{2}{3}}|\leq \delta_{1}$, the continuity of $K$ and that we work on a compact set, to deduce 
\begin{align}\label{i1 is small}
 \textup{I}_{1}\leq C\overline{\epsilon}\sup_{t\in[0,T]} M_{0,0}(f_{\epsilon}(t))^{2},
\end{align}
where $\overline{\epsilon}$ was chosen small but arbitrary and $M_{0,0}(f_{\epsilon})$ is uniformly bounded independently of $\epsilon\in(0,1)$. Notice that $\delta_{1}$ was chosen to be sufficiently small to satisfy both (\ref{main theo isoperimetric line}) and (\ref{i1 is small}). 

Step 1. b) Let $\overline{f},F$ be as in Lemma \ref{lemma dirac measure}. We prove now that
\begin{align}\label{term 1}
\int_{\{a\in\textup{T}(v)\}}\int_{\{a'\in\textup{T}(v')\}}K(c_{0}v^{\frac{2}{3}},v,c_{0}v'^{\frac{2}{3}},v')f_{\epsilon}(\eta,t)f_{\epsilon}(\eta',t)\der\eta'\der\eta
\end{align}
converges to
\begin{align}\label{term 2}
\int_{(0,\infty)}\int_{(0,\infty)}K(c_{0}v^{\frac{2}{3}},v,c_{0}v'^{\frac{2}{3}},v')F(v,t)F(v',t)\der v'\der v
\end{align}
as $\epsilon\rightarrow 0$.

This is done again by splitting the interval $(0,
\infty)^{2}$ into three regions, namely $\{v\leq \frac{1}{M}\}$, $\{v\geq M\}$ and $\{v\in[\frac{1}{M},M]\}$. For the regions $\{v\leq \frac{1}{M}\}$ and $\{v\geq M\}$, we can use moment estimates to prove that the terms (\ref{term 1}) and (\ref{term 2}) approach zero as $M\rightarrow\infty$. In the region $\{(v,v')\in[\frac{1}{M},M]^{2}\}$, we use that $f_{\epsilon}f_{\epsilon}\rightharpoonup \overline{f}  \textup{ } \overline{f}$ as $\epsilon\rightarrow 0$.

\textbf{Step 2.} We continue by proving that the remaining integrals in (\ref{main one dim}) converge to zero as $\epsilon\rightarrow 0$, for every $t\geq \overline{\sigma}$. We first prove that
\begin{align*}
\int_{\{a\in\textup{T}(v)\}}\int_{\{a'>c_{0}v'^{\frac{2}{3}}+\delta_{1}\}}K(a,v,a',v')f_{\epsilon}(\eta,t)f_{\epsilon}(\eta',t)\der \eta \der \eta'\rightarrow 0 \textup{ as } \epsilon\rightarrow 0.
\end{align*}
We have that there exists a constant $C>0$, independent of $\epsilon\in(0,1)$ such that $M_{0,-\alpha}(f_{\epsilon})+M_{0,\beta}(f_{\epsilon})\leq C$. 
Combining these moment bounds with Proposition \ref{solutions stay close to the isoperimetric line}, we obtain that
\begin{align*}
\int_{\{a\in\textup{T}(v)\}}\int_{\{a'>c_{0}v'^{\frac{2}{3}}+\delta_{1}\}}&K(a,v,a',v')|\tilde{\varphi}(v+v')-\tilde{\varphi}(v)-\tilde{\varphi}(v')|f_{\epsilon}(\eta,t)f_{\epsilon}(\eta',t)\der \eta \der \eta'\\
&\lesssim 3||\tilde{\varphi}||_{\infty} \int_{\{a\in\textup{T}(v)\}}v^{\beta}f_{\epsilon}(\eta,t)\der \eta\int_{\{a'>c_{0}v'^{\frac{2}{3}}+\delta_{1}\}}v'^{-\alpha}f_{\epsilon}(\eta',t) \der \eta'\\
&+3||\tilde{\varphi}||_{\infty} \int_{\{a\in\textup{T}(v)\}}v^{-\alpha}f_{\epsilon}(\eta,t)\der \eta\int_{\{a'>c_{0}v'^{\frac{2}{3}}+\delta_{1}\}}v'^{\beta}f_{\epsilon}(\eta',t) \der \eta'\leq C\delta_{2},
\end{align*}
where $\delta_{2}$ was taken as in (\ref{main theo isoperimetric line}). More precisely, in order to prove that
\begin{align}\label{last step needed}
\int_{\{a'>c_{0}v'^{\frac{2}{3}}+\delta_{1}\}}(v'^{\beta}+v'^{-\alpha})f_{\epsilon}(\eta',t) \der \eta'\leq C\delta_{2},
\end{align}
we notice that for the region $\{a'>c_{0}v'^{\frac{2}{3}}+\delta_{1}\}\cap[\frac{1}{M},M]^{2}$, for some sufficiently large $M>1$, we can use Proposition \ref{solutions stay close to the isoperimetric line} to deduce that the integral gives a sufficiently small contribution and then control the regions $\{a>M\}\cup\{a<\frac{1}{M}\}\cup\{v>M\}\cup\{v<\frac{1}{M}\}$ using moment estimates.

\textbf{Step 3.}
Lastly, we prove that, for every $t\geq\overline{\sigma}$, we have that
\begin{align*}
\int_{\{a>c_{0}v^{\frac{2}{3}}+\delta_{1}\}}\int_{\{a'>c_{0}v'^{\frac{2}{3}}+\delta_{1}\}}K(a,v,a',v')f_{\epsilon}(\eta,t)f_{\epsilon}(\eta',t)\der \eta \der \eta'\rightarrow 0 \textup{ as } \epsilon\rightarrow 0.
\end{align*}
As before, we obtain that
\begin{align*}
&\int_{\{a>c_{0}v^{\frac{2}{3}}+\delta_{1}\}}\int_{\{a'>c_{0}v'^{\frac{2}{3}}+\delta_{1}\}}K(a,v,a',v')f_{\epsilon}(\eta,t)f_{\epsilon}(\eta',t)\der \eta \der \eta'\\
&\lesssim 2\int_{\{a>c_{0}v^{\frac{2}{3}}+\delta_{1}\}}\int_{\{a'>c_{0}v'^{\frac{2}{3}}+\delta_{1}\}}v^{-\alpha}v'^{\beta}f_{\epsilon}(\eta,t)f_{\epsilon}(\eta',t)\der \eta \der \eta'\\
&\leq 2\int_{\{a>c_{0}v^{\frac{2}{3}}+\delta_{1}\}}v^{-\alpha}f_{\epsilon}(\eta,t)\der \eta\int_{\{a'>c_{0}v'^{\frac{2}{3}}+\delta_{1}\}}v'^{\beta}f_{\epsilon}(\eta',t) \der \eta'\lesssim \delta_{2}^{2},
\end{align*}
where for the last inequality we used the same argument as in (\ref{last step needed}). This concludes our proof.
\end{proof}
\begin{proof}[Proof of Theorem \ref{pass back to standard coagulation}] Due to Proposition \ref{behavior near initial data}, we can let $\overline{\sigma}\rightarrow 0$ in (\ref{equation with sigma}) using the fact that there exists $C(T)>0$ such that
\begin{align*}
    \sup_{t\in[0,T]}\int_{(0,\infty)}(v^{-1}+v^{2})F(v,t)\der v\leq C(T)
\end{align*}
due to the way $F$ was constructed (see (\ref{form of limit}), (\ref{initial datum})).
\end{proof}
\section{The case of negligible fusion}\label{the case n goes to infinity}
We now consider the form of solutions of (\ref{1 over epsilon strongfusioneq}) if we assume that $\Lambda\rightarrow\infty$.
\begin{prop}
Let $K:(0,\infty)^{4}\rightarrow [0,\infty)$ be a continuous kernel satisfying (\ref{kersym1}), 
(\ref{lower_bound_kernel}) and (\ref{alpha non neg}). Assume the fusion kernel $r\in\textup{C}^{1}(\mathbb{R}_{>0}^{2})$ satisfies (\ref{fusion_form}) and (\ref{ode_fusion}) with $\mu>0$. Assume in addition that $\int_{(0,\infty)^{2}}(v^{-\mu-2}+v^{\mu+2}+v^{\sigma(\mu+2)}+a^{\mu+2})f_{\textup{in}}(a,v)\der v\der a<\infty$. Let $T>0$. Then we can construct $f_{\epsilon}$ as in Definition \ref{definitiontimedependent} for equation (\ref{epsilonic_convergence n case}), for every $\epsilon\in(0,1)$. For this sequence, we have that there exists a constant $C(T)>0$, which is independent of $\epsilon\in(0,1)$, such that
\begin{align}\label{nicemomentestimates n case}
   \sup_{t\in[0,T]} \int_{(0,\infty)^{2}}(v^{-\mu-2}+v^{\mu+2}+v^{\sigma(\mu+2)}+a^{\mu+2})f_{\epsilon}(a,v,t)\der v\der a\leq C(T)
\end{align}
and that there exists a subsequence (which we do not relabel) and $\underline{f}\in\textup{C}([0,T];  \mathscr{M}^{I}_{+}(\mathbb{R}_{>0}^{2}))$ such that $f_{\epsilon}(t)\rightarrow\underline{f}(t)$ as $\epsilon\rightarrow 0$ in the sense of measures, for every $t\in[0,T]$.
\end{prop}
\begin{proof}
Existence of solutions $f_{\epsilon}$ which satisfy (\ref{nicemomentestimates n case}) was proven in Proposition \ref{existence of solutions for the equation}. We are left to prove equicontinuity in time in order to conclude our proof. 

Let $M>1$. It suffices to prove equicontinuity for a fixed function $\varphi\in\compactfundif$ such that $\supp\varphi\in[\frac{1}{M},M]^{2}$. In order to pass to functions $\varphi\in\compactfun$ such that $\supp\varphi\in[\frac{1}{M},M]^{2}$,  we can use that there exists $\varphi_{n}\in\compactfundif$ such that $||\varphi-\varphi_{n}||_{\infty}\leq \frac{1}{n}$. In order to extend to functions $\varphi\in\vanishfun$, we can then use moment estimates by letting $M$ become sufficiently large.

Fix $\varphi\in\compactfundif$. Let $s,t\in[0,T]$. Assume without loss of generality $s\leq t$. Then
\begin{align}
    &\bigg|\int_{(0,\infty)^{2}}[f_{\epsilon}(\eta,t)-f_{\epsilon}(\eta,s)]\varphi(\eta)\der \eta\bigg|\leq\nonumber\\
    &\lesssim 3||\varphi||_{\infty}\int_{s}^{t}M_{0,-\alpha}(f_{\epsilon}(z))M_{0,\beta}(f_{\epsilon}(z))\der z+\epsilon\int_{s}^{t}\int_{(0,\infty)^{2}}a^{\mu}v^{\sigma}(a+c_{0}v^{\frac{2}{3}})f_{\epsilon}(\eta,z)|\partial_{a}\varphi(\eta)|\der \eta\der z\nonumber\\
    &\leq 3||\varphi||_{\infty}\int_{s}^{t}M_{0,-\alpha}(f_{\epsilon}(z))M_{0,\beta}(f_{\epsilon}(z))\der z+C||\partial_{a}\varphi||_{\infty}\int_{s}^{t}\big(M_{\mu+2,0}(f_{\epsilon}(z))+M_{0,\sigma(\mu+2)}(f_{\epsilon}(z))\big)\der z\nonumber\\
    &\leq C|t-s|.
\end{align} 
We deduce then from Arzelà–Ascoli theorem that there exists a subsequence of $\{f_{\epsilon}\}$, which we do not relabel, and an $\underline{f}\in\textup{C}([0,T];\mathscr{M}^{I}_{+}(\mathbb{R}_{>0}^{2})),$ such that $f_{\epsilon}(t)$ converge to $\underline{f}(t)$ in the weak-$^{\ast}$ topology as $\epsilon\rightarrow 0$, for every $t\in[0,T].$
\end{proof}
\begin{proof}[Proof of Theorem \ref{case of slow fusion}]
In order to estimate the fusion term, notice that
\begin{align*}
\epsilon \bigg|\int_{(0,\infty)^{2}}r(\eta)(c_{0}v^{\frac{2}{3}}-a)f_{\epsilon}(\eta,s)\partial_{a}\varphi(\eta,s)\der \eta\bigg|&\lesssim \epsilon ||\partial_{a}\varphi||_{\infty}\sup_{t\in[0,T]}\big(M_{\mu+2,0}(f_{\epsilon}(t))+M_{0,\sigma(\mu+2)}(f_{\epsilon}(t))\big)\\
    &\leq \epsilon C(T)\rightarrow 0 \textup{ as }\epsilon\rightarrow 0.    \end{align*}
To prove that, for fixed $\varphi\in\compactfundif$, the term with the coagulation kernel in (\ref{epsilonic_convergence n case}) converges to the term with the coagulation kernel in (\ref{two dim final}), standard methods used in the study of coagulation equation are used and hence we omit the rest of the proof.
\end{proof}
\begin{prop}\label{extension compact support n infinity}
Let $\mu>0$. Then we can test equation (\ref{two dim final}) with $\varphi(\eta)\equiv a$ or $\varphi(\eta)\equiv v$ if
\begin{align}\label{moment estimates to lose compact support n infinity}
    \sup_{t\in[0,T]} \int_{(0,\infty)^{2}}(a^{2}+v^{-2}+v^{2})\underline{f}(a,v,t)\leq C(T).
\end{align}
\end{prop}
\begin{proof}
The proof is analogous with the proof of Proposition \ref{extension compact support} using the moment estimates (\ref{moment estimates to lose compact support n infinity}).
\end{proof}
\begin{proof}[Proof of Remark \ref{remark conservation mass}]
From Proposition \ref{extension compact support n infinity} we can test equation (\ref{two dim final}) with $\varphi(\eta)\equiv a$ or $\varphi(\eta)\equiv v$ and so the desired quantities are preserved since
\begin{align*}
    \langle\mathbb{K}[\underline{f}],a\rangle=0 \textup{ and }   \langle\mathbb{K}[\underline{f}],v\rangle=0.
\end{align*}
\end{proof}
\subsection*{Acknowledgements}
 The authors gratefully acknowledge the financial support of the collaborative
research centre The mathematics of emerging effects (CRC 1060, Project-ID 211504053) and Bonn International Graduate School of Mathematics at the Hausdorff Center for Mathematics (EXC 2047/1, Project-ID 390685813) funded through the Deutsche Forschungsgemeinschaft (DFG, German Research Foundation). The funders had no role in study design, analysis, decision to publish, or preparation of the manuscript.

\subsubsection*{Statements and Declarations}

\textbf{Conflict of interest} The authors declare that they have no conflict of interest.

\textbf{Data availability} Data sharing not applicable to this article as no datasets were generated or analysed during the current study.

\printbibliography

\end{document}